\documentclass[10pt,a4paper]{article}

\usepackage{anysize}
\marginsize{3cm}{3cm}{2cm}{2cm}

\usepackage{amsmath,amssymb,amsthm}
\usepackage{enumitem}

\newcommand{\R}{\mathbb{R}}

\newcommand{\N}{\mathbb{N}}
\newcommand{\Z}{\mathbb{Z}}
\renewcommand{\S}{\mathcal{S}}
\newcommand{\F}{\mathcal{F}}

\renewcommand{\i}{\mathbf{i}}

\theoremstyle{plain}
\newtheorem{thm}{Theorem}[section]
\newtheorem{lma}[thm]{Lemma}
\newtheorem{cor}[thm]{Corollary}
\newtheorem{prop}[thm]{Proposition}

\theoremstyle{definition}
\newtheorem{df}[thm]{Definition}

\theoremstyle{remark}
\newtheorem*{rem}{Remark}
\theoremstyle{remark}
\newtheorem*{ex}{Example}

\newcommand{\supp}{\mathop{\mathrm{supp}}}
\newcommand{\sgn}{\mathop{\mathrm{sgn}}}

\begin{document}

\title{Superposition in Modulation Spaces with Ultradifferentiable Weights
}

\author{Maximilian Reich \\
\small Fakult\"at f\"ur Mathematik und Informatik \\[-0.8ex]
\small TU Bergakademie Freiberg, Germany \\[-0.8ex]
\small \texttt{maximilian.reich@math.tu-freiberg.de}\\
}



\date{\today}
		
\maketitle
		
\begin{abstract}
	In the theory of nonlinear partial differential equations we need to explain superposition operators. For modulation spaces equipped with particular ultradifferentiable weights this was done in \cite{rrs}. In this paper we introduce a class of general ultradifferentiable weights for modulation spaces $\mathcal{M}^{w_*}_{p,q}(\R^n)$ which have at most subexponential growth. We establish analytic as well as non-analytic superposition results in the spaces $\mathcal{M}^{w_*}_{p,q}(\R^n)$. 
\end{abstract}

	
	\section{Introduction}

	
	Classical modulation spaces got originally introduced by Feichtinger in \cite{feichtingerGroup} as a family of Banach spaces controlling globally local frequency information of a function. Thus, modulation spaces are an important tool when discussing problems in time-frequency analysis. But it also turned out that modulation spaces find fruitful applications in the theory of partial differential equations (e.g. see \cite{cordero, iwabuchi, ruzhansky, qiang, wang1, wangNonlinear, wang, wangExp}) and in the theory of pseudo-differential calculus and Fourier integral operators (e.g. see \cite{toft, toftCubo, heil, zhong}). \\
	In the study of partial differential equations Gevrey analysis is an effective tool, where models are treated in the phase space instead of the physical space. Therefore functions are characterized by its behavior on the Fourier transform side. Here we define Gevrey functions as follows:
	\[ f\in \mathcal{G}_s (\R^n) \Longleftrightarrow e^{\langle \xi \rangle^{\frac{1}{s}}} \F f(\xi) \in L^2(\R^n), \]
	where $s\geq 1$. Let us consider the strictly hyperbolic Cauchy problem
	\begin{equation} \label{Gevrey}
		u_{tt} - a(t)u_{xx} = 0, \quad u(0,x) = \phi(x), \quad u_t(0,x) = \psi(x)
	\end{equation}
	with initial data $(\phi, \psi) \in \mathcal{G}_s (\R^n) \times \mathcal{G}_s (\R^n)$ and positive H\"older coefficient $a=a(t) \in C^\alpha [0,T]$, $0<\alpha<1$. In \cite{colombini} the authors showed that \eqref{Gevrey} is globally (in time) well-posed in the Gevrey space $\mathcal{G}_s (\R^n)$ for $s < \frac{1}{1-\alpha}$. In \cite{brs} the authors proved a well-posedness result of the corresponding semi-linear Cauchy problem 
	\begin{equation} \label{GevreyNL}
		u_{tt} - a(t)u_{xx} = f(u), \quad u(0,x) = \phi(x), \quad u_t(0,x) = \psi(x)
	\end{equation}
	with an admissible nonlinearity $f=f(u)$ and Gevrey data $\phi$ and $\psi$. Hence, the obstacle is to explain superposition operators $T_f$ defined by
	\begin{equation} \label{SupOperator}
		T_f: u \mapsto T_f u := f(u)
	\end{equation}
	in corresponding function spaces. This was the main motivation in \cite{rrs} in order to achieve superposition results by employing similar strategies as in \cite{brs}. As described before functions from scales of modulation spaces are also defined by means of their behavior on the Fourier transform side. \\
	In the present paper the goal is to treat analytic as well as non-analytic superposition in modulation spaces equipped with a general class of weights growing faster than any polynomial but at most subexponentially, i.e., a natural extension of results obtained in \cite{rrs}. \\
	The paper is organized as follows. First of all we define modulation spaces $M^s_{p,q} (\R^n)$ by means of a uniform decomposition method as used in, e.g., \cite{feichtingerGroup}, \cite{trzaa} and \cite{wangExp} (see Section \ref{ClassicalMod}). \\
	In Section \ref{WeightClass} we give a formal definition of modulation spaces $M^{w_*}_{p,q} (\R^n)$ equipped with weights $e^{w_*(|\cdot|)}$. Then an appropriate class $\mathcal{W} (\R)$ of general weights $w_*$ is proposed and some basic properties of modulation spaces $M^{w_*}_{p,q} (\R^n)$ are discussed. \\
	By proving algebra properties in Section \ref{MultAlg} analytic nonlinearities can be handled. \\
	In Section \ref{NonAnalSup} we are able to find non-analytic functions $f \in C^\infty (\R^n)$ such that the corresponding superposition operator $T_f$ explained by \eqref{SupOperator} maps modulation spaces $M^{w_*}_{p,q} (\R^n)$ into itself. \\
	Some open problems and concluding remarks complete the paper (see Section \ref{ConcRem}).
	
	\section{Modulation Spaces} \label{ClassicalMod}
	
	First of all we introduce some basic notation and definitions. In $\R^n$ the notation of 
multi-indices $\alpha = (\alpha_1, \ldots , \alpha_n)$ is used, where $|\alpha|=\sum_{j=1}^n \alpha_j$.
 Given two multi-indices $\alpha$ and $\beta$, then $\alpha \leq \beta$ means $\alpha_j \leq \beta_j$ for $1\leq j\leq n$. 
Furthermore, let $f$ be a function on $\R^n$ and $x\in\R^n$, then
	\[ x^\alpha = \prod_{j=1}^n x_j^{\alpha_j} \]
	and
\[ D^\alpha f(x) = \frac{1}{\i^{|\alpha|}} \frac{\partial^\alpha}{\partial x^\alpha} f(x) = 
\frac{1}{\i^{|\alpha|}} \Big(\prod_{j=1}^n \frac{\partial^{\alpha_j}}{\partial x_j^{\alpha_j}} \Big)f(x). 
\]
	A function $f\in C^\infty (\R^n)$ belongs to the Schwartz space $\S(\R^n)$ if and only if
		\[ \sup_{x\in \R^n} | x^\alpha D^\beta f(x)| < \infty \]
		for all multi-indices $\alpha, \beta$. The set of all tempered distributions is denoted by $\S'(\R^n)$ which is the dual space of 
		$\S(\R^n)$. Moreover, by $C^\infty_0 (\R^n)$ we denote the space of smooth functions with compact support. \\
We introduce $\langle \xi\rangle_m := (m^2+|\xi|^2)^{\frac{1}{2}}$.
		If $m=1$, then we write $\langle \xi \rangle$ for simplicity. \\
		The notation $a \lesssim b$ is equivalent to $a \leq Cb$ with a positive constant $C$. Moreover, by writing $a \asymp b$ we mean $a \lesssim b \lesssim a$. Let $X$ and $Y$ be two Banach spaces. Then the symbol $X \hookrightarrow Y$ indicates that the embedding is continuous. The Fourier transform of 
an admissible function $f$ is defined by
		\[ \F f(\xi) = \hat{f} (\xi) = (2\pi)^{-\frac{n}{2}} \int_{\R^n} f(x) e^{-\i x\cdot \xi} \, dx \quad (x,\xi \in\R^n). \]
		Analogously, the inverse Fourier transform is defined by
		\[ \F^{-1} \hat{f}(x) = (2\pi)^{-\frac{n}{2}} \int_{\R^n} \hat{f}(\xi) e^{\i x\cdot \xi} \, d\xi \quad (x,\xi \in\R^n). \]		
		
		First in \cite{feichtingerGroup} Feichtinger defined modulation spaces by taking Lebesgue norms of the so-called short-time Fourier transform of a function $f$ with respect to $x$ and $\xi$. The short-time Fourier transform is a particular joint time-frequency representation. For the definition and mapping properties we refer to \cite{groechenig}. By introducing the following decomposition principle we adopt the idea of obtaining local frequency properties of a function $f$. Related frequency decomposition techniques are explained in \cite{groechenig}. A special case, the so-called frequency-uniform decomposition, was independently introduced by Wang (e.g., see \cite{wang}). Let $\rho: \R^n \mapsto [0,1]$ be a Schwartz function which 
is compactly supported in the cube 
\[
Q_0 := \{ \xi \in \R^n: -1\leq \xi_i \leq 1,\:  i=1,\ldots,n \}\, .
\] 
Moreover, 
\[
\rho(\xi)=1 \qquad \mbox{if}  \quad  |\xi_i|\leq \frac{1}{2}\, , \qquad i=1, 2, \ldots \, n. 
\]
With $\rho_k (\xi) :=\rho (\xi-k) $, $\xi \in \R^n$, $k \in \Z^n$, it follows
\[
\sum_{k \in \Z^n} \rho_k (\xi)\ge 1 \qquad \mbox{for all}\quad \xi \in \R^n\, .
\]
Finally, we define
\[ 
\sigma_k(\xi) := \rho_k(\xi) \Big(\sum_{k\in\Z^n} \rho_k(\xi)\Big)^{-1}, \qquad \xi \in \R^n\, , \quad k\in\Z^n \, .
\]
The following  properties are obvious: 
		\begin{itemize}
			\item $ 0 \le \sigma_k(\xi) \le 1$ for all $\xi \in \R^n$;
			\item $\supp \sigma_k \subset Q_k := \{ \xi \in \R^n: -1\leq \xi_i -k_i \leq 1, \: i=1,\ldots,n \} $;
			\item $\displaystyle \sum_{k\in\Z^n} \sigma_k(\xi) \equiv 1$ for all $\xi\in\R^n$;
			\item there exists a constant $C>0$ such that $\sigma_k (\xi) \ge C$ if $ \max_{i=1, \ldots \, n}\, |\xi_i-k_i|\leq \frac{1}{2}$; 
			\item for all $m \in \N_0$ there exist positive constants $C_m$ such that for $|\alpha|\leq m$
\[
\sup_{k \in \Z^n}\, \sup_{\xi \in \R^n} \, |D^\alpha \sigma_k(\xi)|\leq C_m\,   .
\]
		\end{itemize}
		The operator
		\[ \Box_k := \F^{-1} \left( \sigma_k \F (\cdot) \right), \quad k\in\Z^n, \]
		is called uniform decomposition operator. \\
		In \cite{feichtingerGroup} Feichtinger also showed that there is an equivalent definition of modulation spaces by means of the uniform decomposition operator. 
		
		\begin{df} \label{defdecomp}
			Let $1\leq p,q \leq \infty$ and assume $s \in\R$ to be the weight parameter. Suppose the window $\rho\in \S(\R^n)$ is compactly supported. Then the weighted modulation space $M_{p,q}^{s}(\R^n)$ consists of all tempered distributions $f\in \S'(\R^n)$ such that their norm 
			\[ \|f\|_{M_{p,q}^{s}} = \Big( \sum_{k\in\Z^n} \langle k \rangle^{sq} \|\Box_k f\|_{L^p}^q \Big)^{\frac{1}{q}} \]
			is finite with obvious modifications when $p=\infty$ and/or $q=\infty$.
		\end{df}
		
		\begin{rem}
			General references with respect to (weighted) modulation spaces are \cite{feichtingerGroup}, \cite{groechenig}, \cite{toftConv}, \cite{trzaa} and \cite{wangExp} to mention only a few.
		\end{rem}

	\section{Modulation Spaces with General Weights} \label{WeightClass}
	The goal of this section is to define a set $\mathcal{W} (\R)$ of admissible weight functions for modulation spaces in order to obtain analytic as well as non-analytic superposition results in Sections \ref{MultAlg} and \ref{NonAnalSup}, respectively. 
	
	\subsection{Definition and Basic Properties}
	
	We give a formal definition of the weighted modulation space $\mathcal{M}^{w_*}_{p,q}(\R^n)$.
	
	\begin{df} \label{ModSpace}
		Let $1\leq p,q \leq \infty$ and let $w_*$ be a monotonically increasing weight function with $w_* (0) \geq 0$. Suppose the window $\rho\in C^\infty_0 (\R^n)$ to be a fixed window function. By $(\sigma_k)_k$ we denote the associated uniform decomposition of unity. Then the weighted modulation space $\mathcal{M}^{w_*}_{p,q}(\R^n)$ is the collection of all $f\in L^p(\R^n)$ such that
		\[ \| f\|_{\mathcal{M}^{w_*}_{p,q}} := \bigg( \sum_{k\in\Z^n} e^{q w_*(|k|)} \| \Box_k f \|^q_{L^p} \bigg)^{\frac{1}{q}} < \infty \]
		(with obvious modifications when $p=\infty$ and/or $p=\infty$). 
	\end{df}
	
	\begin{rem}
		If we put $w_* (|k|) := s \log \langle k \rangle$, then Definition \ref{ModSpace} coincides with Definition \ref{defdecomp} in the sense of equivalent norms. The cases $w_*(|k|) = \langle k \rangle^{\frac{1}{s}}$, $s>1$, and $w_*(|k|) = \log \langle k\rangle_{e^e} \log \log \langle k \rangle_{e^e}$ were considered in \cite{rrs}, respectively. 
	\end{rem}
	
	Let us mention some basic properties of modulation spaces $\mathcal{M}^{w_*}_{p,q} (\R^n)$, for the proofs see \cite{rrs}.

\begin{lma}\label{basic}
(i) The modulation space $\mathcal{M}^{w_*}_{p,q}(\R^n)$ is a Banach space.
\\
(ii) $\mathcal{M}^{w_*}_{p,q}(\R^n)$ is independent of the choice of the window $\rho\in C_0^\infty (\R^n)$ in the sense of equivalent norms.
\\
(iii) $\mathcal{M}^{w_*}_{p,q}(\R^n)$ has the Fatou property, i.e., if $(f_m)_{m=1}^\infty \subset \mathcal{M}^{w_*}_{p,q}(\R^n)$ 
is a sequence such that 
$ f_m \rightharpoonup f $ (weak convergence in $\S'(\R^n)$) and 
\[
 \sup_{m \in \N}\,  \|\, f_m \, \|_{\mathcal{M}^{w_*}_{p,q}} < \infty\, , 
\]
then $f \in \mathcal{M}^{w_*}_{p,q}$ follows and 
\[
\|\, f \, \|_{\mathcal{M}^{w_*}_{p,q}} \le \sup_{m\in \N}\,  \|\, f_m \, \|_{\mathcal{M}^{w_*}_{p,q}} < \infty\, .
\]
\end{lma}

Moreover, the spaces $\mathcal{M}^{w_*}_{p,q}(\R^n)$ are monotone in $p$ and $q$ for a fixed weight $w_*$, where we stress Nikol'skij's inequality, see, e.g.,  Nikol'skij \cite[3.4]{Ni} or Triebel \cite[1.3.2]{triebel}.

\begin{lma}\label{einbettung}
	Let $1 \leq p_0 < p \leq \infty$ and $1 \leq q_0 < q \leq \infty$. Then for fixed $w_*\in \mathcal{W} (\R)$ the following  embeddings hold and are continuous:
	\[ \mathcal{M}^{w_*}_{p_0,q}(\R^n) \hookrightarrow \mathcal{M}^{w_*}_{p,q}(\R^n) \]
	and
	\[ \mathcal{M}^{w_*}_{p,q_0}(\R^n) \hookrightarrow \mathcal{M}^{w_*}_{p,q}(\R^n)\, ; \]
	i.e., for all $1\le p,q\le \infty$ we have 
	\[ \mathcal{M}^{w_*}_{1,1}(\R^n) \hookrightarrow \mathcal{M}^{w_*}_{p,q}(\R^n) \hookrightarrow  \mathcal{M}^{w_*}_{\infty,\infty}(\R^n)\, . \]
\end{lma}

	\subsection{A Class of General Weight Functions}

	Subsequently we introduce classes of functions which will turn out to be appropriate in order to prove our desired results in the Sections \ref{MultAlg} and \ref{NonAnalSup}, respectively.
	
	\begin{df} \label{SlowlyVF}
		A measurable function $f: (0,\infty) \mapsto (0,\infty)$ is called slowly varying if 
		\[ \lim_{t\to \infty} \frac{f(tx)}{f(t)} = 1 \]
		for all $x>0$. 
	\end{df}
	For Definition \ref{SlowlyVF} we refer to the definition in Chapter 1.2 in \cite{bgt}. By the Uniform Convergence Theorem, see \cite[Theorem 1.2.1]{bgt}, it follows that the convergence in Definition \ref{SlowlyVF} is uniform with respect to $x$ in any fixed interval $[a,b]$, $0<a<b<\infty$. Moreover, we recall the so-called Karamata representation theorem, see \cite[1.3.1]{bgt}. 
	\begin{thm} \label{karamata}
		A function $f$ is slowly varying if and only if there exists a positive number $B$ such that 
		\begin{equation} \label{ReprSVF}
			f(x) = \exp \Big\{ \eta(x) + \int_B^x \frac{\varepsilon(t)}{t} \, dt \Big\} \, , \qquad x\geq B,
		\end{equation}
		where $\eta(x) \to c$ $(|c|<\infty)$ as $x\to \infty$ and $\varepsilon(t) \to 0$ as $t\to \infty$. Both functions $\eta$ and $\varepsilon$ are measurable and bounded. 
	\end{thm}

	\begin{rem}
		For a more detailed insight see \cite{bgt} and \cite{gs}. Here we remark one important fact. Since the functions $f$, $\eta$ and $\varepsilon$ are characterized by their behavior at infinity we can alter them on finite intervals as long as they remain measurable und bounded, respectively. Thus, without loss of generality we can choose the number $B$ in Theorem \ref{karamata} to be any desired positive number by adjusting $\varepsilon$. \\ 
		Generally spoken slowly varying functions resemble functions converging at infinity. For several properties of slowly varying functions we refer to Proposition 1.3.6 in \cite{bgt}. For instance,  any real power of a slowly varying function is again slowly varying.
	\end{rem}
	
	Now let us consider functions which increase faster than slowly varying functions. In fact, the so-called regularly varying functions resemble a monomial near infinity.
	\begin{df} \label{RegularlyVF}
		A measurable function $f: (0,\infty) \mapsto (0,\infty)$ is called regularly varying if 
		\[ \lim_{t\to \infty} \frac{f(tx)}{f(t)} = x^\alpha \]
		for all $x>0$ and $\alpha\in\R$. The number $\alpha$ is called index.
	\end{df}
	
	\begin{rem}
		For the definition see Chapter 1.4.2 in \cite{bgt}. \\
		Again a Uniform Convergence Theorem, see \cite[Theorem 1.5.2]{bgt}, states the uniform convergence in Definition \ref{RegularlyVF} with respect to $x$ on each interval $(0,b]$, $0<b<\infty$, if index $\alpha>0$. This is sufficient for our purposes since we will only work with regularly varying functions with nonnegative index. \\
		For some properties of regularly varying functions we refer to Proposition 1.5.7 in \cite{bgt}. 
	\end{rem}

	By stressing Theorem 1.4.1 in \cite{bgt}, the so-called Characterisation Theorem, we are able to find a connection between regularly varying and slowly varying functions. If $f$ is a regularly varying function of index $\alpha$, then $f(x) = x^\alpha g(x)$ with a slowly varying function $g$. In terms of Theorem \ref{karamata} we find the representation
	\[ f(x) = x^\alpha \exp \Big\{ \eta(x) + \int_B^x \frac{\varepsilon(t)}{t} \, dt \Big\} \, , \qquad x\geq B. \]
	Now we can also assume some regularity for the function $f$. In fact, Theorem 1.3.3 in \cite{bgt} yields the asymptotical equivalence of an arbitrary slowly varying function $g$ and a smooth slowly varying function $g_1$, i.e., $\lim_{t\to \infty} g(t)/g_1(t) = 1$, where we also replaced the function $\eta=\eta(x)$ by its limit $c$ in the representation of $g_1$. 
Therefore let $f \in C^\infty (0, \infty)$ be a regularly varying function with the representation
	\begin{equation} \label{CharRVF}
		f(x) = x^\alpha \exp \Big\{ c + \int_B^x \frac{\varepsilon(t)}{t} \, dt \Big\} \, , \qquad x\geq B.
	\end{equation}
	Note that in \eqref{CharRVF} we modified the function $\varepsilon=\varepsilon(t)$ in consequence of the assumed regularity on $f$.  A further modification of $\varepsilon$ admits an appropriate choice of the positive number $B$. It follows that
	\begin{equation} \label{CharRVF2}
		\lim_{x\to\infty} x \frac{f'(x)}{f(x)} = \alpha
	\end{equation}
	since
	\[	x \frac{f'(x)}{f(x)} = \alpha + \varepsilon(x), \qquad \quad x\geq B. \]
	Note that in \eqref{CharRVF} we did not change the behavior of $\varepsilon$ at infinity, i.e., the function $\varepsilon = \varepsilon(x)$ tends to $0$ as $x\to \infty$. All in all we have shown that the Definitions \ref{SlowlyVF} and \ref{RegularlyVF} do not determine the regularity of functions, i.e., a function satisfying Definition \ref{SlowlyVF} or Definition \ref{RegularlyVF} can be of arbitrary smoothness, respectively. \\	 
Now we are ready to define the set $\mathcal{W}(\R)$ of all admissible weight functions $w_*$. Remark that we apply the weight 
\[ w(|k|) := e^{w_*(|k|)} \]
in Definition \ref{ModSpace} of the modulation spaces $\mathcal{M}^{w_*}_{p,q}(\R^n)$. 
\begin{df} \label{DefClassWeight}
	The set $\mathcal{W} (\R)$ contains all functions $w_* = w_*(x)$ , $x\in [0,\infty)$, such that $w_*$ satisfies the following conditions:
	\begin{enumerate}[label=(A\arabic*)]
		\item $w_* \in C^\infty (0, \infty)$ is regularly varying with index $\alpha$ such that $0\leq \alpha<1$, \label{A1}
		\item $w_* \geq 1$, \label{A5}
		\item $w_*$ is strictly increasing, \label{A3}
		\item $\displaystyle \lim_{t \to \infty} ( w_*(t) - M \log t ) = \infty$ for all $M$, \label{A4}
		\item $w'_*(0)=0$, \label{A2}
		\item $w{''}_*$ has finite changes of sign, i.e., there exists a number $\tau$ such that $w{''}_*(x) <0$ for all $x> \tau$. \label{A6}
	\end{enumerate}	
\end{df}

Let $w_* \in \mathcal{W} (\R)$. Due to \ref{A1} and the previous arguments $w_*$ can be represented by \eqref{CharRVF} and satisfies \eqref{CharRVF2}. Thus, for all $\epsilon > 0$ there exists a positive number $x_0$ such that
\begin{enumerate}[label=(B\arabic*)]
	\item $\qquad \qquad \Big| x \frac{w'_*(x)}{w_*(x)} \Big| < \alpha + \epsilon \quad$ for all $x\geq x_0$. \label{A7}
\end{enumerate}
Note that $\alpha \in [0,1)$ in \ref{A1}. If we put $\epsilon := 1-\alpha$, then we can fix the number $x_0$. \\
By considering the derivative of the function $p(x):= \frac{x}{w_*(x)}$ we know that $p=p(x)$ is monotonically increasing for all $x\geq x_0$, where the same $x_0$ is used as in \ref{A7}. However, $p=p(x)$ can attain several local maxima for $x\in [0,x_0]$. Moreover, $p(x) \to \infty$ as $x\to \infty$ due to \ref{A1} and the corresponding representation of regularly varying functions, see \eqref{CharRVF}. Thus, we may take $x_1\geq x_0$ such that 
\begin{enumerate}[label=(B\arabic*)]
	\setcounter{enumi}{1}
	\item $\qquad \qquad \frac{t_0}{w_*(t_0)} \leq  \frac{x_1}{w_*(x_1)} \leq \frac{t_1}{w_*(t_1)}$ \label{A8}
\end{enumerate}
holds for all $t_0 \leq x_1$ and for all $t_1 \geq x_1$. Then we put 
\begin{equation} \label{Xstart}
	\tilde{x} := \max (\tau, 2 x_0, 2 x_1, B),
\end{equation}
where $\tau$ comes from \ref{A6} and $B$ comes from the representation \eqref{CharRVF}. \\
\begin{rem}
	(i) For every fixed weight function $w_* \in \mathcal{W}(\R)$ there exists $\tilde{x}$ as in \eqref{Xstart} such that \ref{A7} and \ref{A8} are satisfied. Here it is essential that $\alpha<1$ in \ref{A1}. If we suppose $\alpha=1$, then the arguments which lead to \ref{A7} and \ref{A8} would fail. However \ref{A7} and \ref{A8} are fundamental in order to prove Lemma \ref{numerGeneral} below. \\
	(ii) In case $\alpha=0$ in condition \ref{A1} we call the weight functions slowly varying, cf. Definition \ref{SlowlyVF}. \\
	(iii) \ref{A5} and \ref{A3} in Definition \ref{DefClassWeight} are natural conditions for weight functions. \\
	(iv) Assumption \ref{A4} ensures that the weights $w(|k|) := e^{w_*(|k|)}$, $w_* \in \mathcal{W}(\R)$, are functions increasing faster than any polynomial. It means that we will treat modulation spaces equipped with ultradifferentiable weights. On the other hand $\alpha <1$ in assumption \ref{A1} indicates that the modulation spaces are equipped with weights of at most subexponential growth. \\
	(v) Conditions \ref{A1}, \ref{A2} and \ref{A6} have technical character. Nevertheless they are reasonable for defining a suitable class of weight functions for modulation spaces. \\
	(vi) By putting $w_* (|k|) = \langle k \rangle^{\frac{1}{s}}$ the Gevrey-modulation space $\mathcal{GM}^s_{p,q} (\R^n)$ introduced in \cite{rrs} coincides with $\mathcal{M}^{w_*}_{p,q} (\R^n)$ for $s>1$. Clearly we have $w_*\in \mathcal{W}(\R)$. Moreover, it holds $\mathcal{UM}^{p,q}(\R^n) = \mathcal{M}^{w_*}_{p,q} (\R^n)$ if $w_*(|k|)= \log \langle k \rangle_{e^e} \log\log \langle k \rangle_{e^e}$, which is obviously contained in the set $\mathcal{W} (\R)$ as well (cf. \cite{rrs}).
\end{rem}

\begin{ex}
	We introduce the notation 
	\[ l_m \langle k \rangle_* := \underbrace{\log \log \ldots \log}_{m\text{-times}} \langle k\rangle_* \, , \]
	i.e., $l_m$ is the $m$ times composition of the $\log$-function with itself. Here $*$ stands for a sufficiently large number (depending on $m$) such that $l_m \langle k \rangle_* \geq 1$ for all $k\in \Z^n$. A natural extension of weight functions mentioned in the previous remark (vi) are functions of the form
	\begin{equation} \label{WeightEx1}
		w_* (|k|) = \langle k\rangle^{\frac{1}{s}} \big( \log\langle k \rangle_* \big)^{r_1} \big( \log\log \langle k \rangle_* \big)^{r_2} \cdot \ldots \cdot \big( l_m \langle k \rangle_* \big)^{r_m} 
	\end{equation}
	with parameters $s\in (1,\infty)$ and $r_k \in \R$, $k=1,2,\ldots, m$. Then $w_* \in \mathcal{W}(\R)$ and $w_*$ is regularly varying of index $\alpha = \frac{1}{s}$. 	Moreover, there exist two cases for $s=\infty$, that is, $w_*$ is a slowly varying function. If we suppose $r_1>1$, then $r_2,r_3,\ldots, r_m \in \R$. If $r_1=1$ and there exists $l\in\N_+$ such that $r_i=0$ for $i=2,3,\ldots, l$, then $r_{l+1}>0$ and $r_{l+2}, r_{l+3}, \ldots, r_m \in \R$. It follows $w_* \in \mathcal{W}(\R)$.
\end{ex}

\subsection{Embedding into Spaces of Ultradifferentiable Functions}

We can show that only very smooth functions have a chance to belong to the space $\mathcal{M}^{w_*}_{p,q} (\R^n)$ if $w_* \in \mathcal{W}(\R)$. Here we apply the theory of so-called ultradifferentiable functions, see, e.g., \cite{bjoerck, bonet, braun, gs, pascu, schindl, triebel}. We notice that $w_* \in \mathcal{W}(\R)$ is a weight function in the sense of \cite{braun}. Now we recall the definition of the space $\mathcal{E}_{\{w_*\}} (\R^n)$ of all $\omega_*$-ultradifferentiable functions of Roumieu type, see e.g. Definition 4.1 in \cite{braun},
\[ \mathcal{E}_{\{w_*\}} (\R^n) := \Big\{ f\in C^\infty (\R^n): \forall \mbox{ compact } K\subset \R^n \,\, \exists m\in\N:\sup_{x\in K, \alpha\in\N^n} \frac{|D^\alpha f (x)|}{e^{\frac{1}{m} \sup_{y>0}(m|\alpha| y - w_*(e^y))}} < \infty \Big\} . \]
So we try to give a characterisation of functions $f\in \mathcal{M}^{w_*}_{p,q} (\R^n)$, $w_* \in \mathcal{W}(\R)$, by certain growing conditions on their derivatives analogously to Corollary 2.11 in \cite{rrs}. 
\begin{prop} \label{UltraCharFunc}
	Let $1\leq p,q \leq \infty$ and $w_* \in \mathcal{W} (\R)$ be a fixed weight function. Then the following continuous embedding holds:  
	\[ \mathcal{M}^{w_*}_{p,q} (\R^n) \hookrightarrow \mathcal{E}_{\{w_*\}} (\R^n) \, . \]
\end{prop}
\begin{proof}
Let $f \in \mathcal{M}^{w_*}_{p,q} (\R^n)$. Following the arguments in the proof of Corollary 2.11 in \cite{rrs} yields
\begin{equation} \label{derIn}
	|D^\alpha f(x)| \lesssim \bigg( \sum_{|k|<N} (1+|k|)^{|\alpha|} e^{-w_*(|k|)} \bigg) \|f\|_{\mathcal{M}^{w_*}_{\infty,\infty}}
\end{equation}
for any $N\in \N$.
Taking account of \ref{A5} observe that
\begin{eqnarray*}
	 \sum_{|k|<N} (1+|k|)^{|\alpha|} e^{-w_*(|k|)} & \lesssim & \int_{\R^n} |\xi|^{|\alpha|} e^{-w_*(|\xi|)} \, d\xi  \\
	 	& \lesssim & \int_{|\xi | >1} |\xi|^{|\alpha|} e^{-w_*(|\xi|)} \, d\xi \\
	 	& \leq & \sup_{|\xi| >1} \big( |\xi|^{|\alpha|} e^{-\frac{w_*(|\xi|)}{2}} \big) \int_{\R^n} e^{-\frac{w_*(|\xi|)}{2}} \, d\xi  \\
		& \lesssim &  \sup_{|\xi| >1} \big( |\xi|^{2|\alpha|} e^{-w_*(|\xi|)} \big)^{\frac{1}{2}} \\
		& = & \sup_{t>1} \big( e^{2|\alpha| \log t -w_*(t)} \big)^{\frac{1}{2}} \\
		& \stackrel{[y = \log t]}{=} & \sup_{y>0} e^{\frac{1}{2} (2|\alpha| y -w_*(e^y))} \, . \\
\end{eqnarray*}
Lemma \ref{einbettung} and the definition of the space $\mathcal{E}_{\{w_*\}} (\R^n)$ yield the desired result.
\end{proof}

A usual technique to describe regularity of a function is to estimate the growth of its derivatives by so-called weight sequences. For this purpose we define the so-called Carleman class of Roumieu type $\mathcal{E}_{\{M_p\}} (\R^n)$, see e.g. Definition 7 in \cite{bonet},
	\[ \mathcal{E}_{\{M_p\}} (\R^n) := \Big\{ f\in C^\infty (\R^n): \forall \mbox{ compact } K\subset \R^n \,\, \exists h>0: \sup_{x\in K, \alpha\in\N^n} \frac{|D^\alpha f (x)|}{h^{|\alpha|} M_{|\alpha|}} < \infty \Big\}, \]
where $(M_p)_{p\in \N}$ is a weight sequence, see Definition 4 in \cite{bonet}. Let $w_* \in \mathcal{W}(\R)$ be a weight function such that $w_*$ is regularly varying of index $0<\alpha <1$. Then $w_*$ is so-called Gelfand-Shilov admissible, see Definition 3.1. in \cite{pascu}. By stressing Definition 3.2. and Proposition 3.1. in \cite{pascu} it follows that the associated sequence $(M_p)_{p\in \N}$ of the weight function $w_*$, which is defined by 
\begin{equation} \label{assocSeq}
	M_p := \sup_{r>0} \Big( \frac{r^p}{e^{w_*(r)}} \Big)
\end{equation}
for all $p\in\N$, satisfies 
\begin{enumerate}[label=\roman*)]
	\item $M_p^2 \leq M_{p-1} M_{p+1}$ for all $p\in\N$, \label{ws-31}
	\item there exists a constant $H\geq 1$ such that $M_{p+q} \leq H^{p+q} M_p M_q$ for all $p,q\in\N$,  \label{ws-32}
	\item there exist $\eta<\frac{1}{2}$, $h>0$ such that $M_p \geq h^p p^{p\eta}$ for all $p\in\N$.  \label{ws-33}
\end{enumerate}
For the sake of normalisation we can also define the associated sequence $(M_p)_{p\in \N}$ of the weight function $w_*$ by 
\begin{eqnarray*}
	N_p & := & \sup_{r>0} \Big( \frac{r^p}{e^{w_*(r)}} \Big) \, ,\\
	M_p & := & N_p / N_0 \, .
\end{eqnarray*}
Then we have the properties \ref{ws-31}-\ref{ws-33} as well as $M_0 = 1$. In fact, $(M_p)_{p\in\N}$ is a weight sequence in the sense of \cite{bonet}. 

\begin{prop} \label{UltraCharSeq}
	Let $1\leq p,q \leq \infty$ and $w_* \in \mathcal{W} (\R)$ be a fixed weight function such that $w_*$ is regularly varying of index $0<\alpha<1$. If $f\in \mathcal{M}^{w_*}_{p,q} (\R^n)$, then $f$ is infinitely often differentiable and there exists a constant $C=C(f,n)$ such that
	\[ |D^\alpha f| \leq C^{|\alpha|+1} M_{|\alpha|}, \]
	where $M_p$, $p\in \N$, is the associated sequence of the weight function $w_*$ defined by \eqref{assocSeq}. \\
	Moreover, it follows $f\in \mathcal{E}_{\{M_p\}} (\R^n)$.
\end{prop}
\begin{proof}
Starting point is \eqref{derIn}. Taking account of \eqref{assocSeq} and the corresponding properties of the sequence $M_p$, $p\in \N$, it follows
\begin{eqnarray*}
	 \sum_{|k|<N} (1+|k|)^{|\alpha|} e^{-w_*(|k|)} & \lesssim & \int_{\R^n} |\xi|^{|\alpha|} e^{-w_*(|\xi|)} \, d\xi  \\
	 	& \leq & \sup_{\xi\in\R^n} \big( |\xi|^{|\alpha|} e^{-\frac{w_*(|\xi|)}{2}} \big) \int_{\R^n} e^{-\frac{w_*(|\xi|)}{2}} \, d\xi  \\
		& \lesssim &  \sup_{\xi\in\R^n} \big( |\xi|^{2|\alpha|} e^{-w_*(|\xi|)} \big)^{\frac{1}{2}} \\
		& = & \big( M_{2|\alpha|} \big)^{\frac{1}{2}} \\
	 	& \leq & \big( H^{2|\alpha|} M^2_{|\alpha|} \big)^{\frac{1}{2}} \\
		& = & H^{|\alpha|} M_{|\alpha|}
\end{eqnarray*}
with a constant $H\geq 1$ coming from the estimate \ref{ws-32} above. The proof is complete.
\end{proof}

\begin{rem}
We shall give a remark to the restriction on $w_*$ in Proposition \ref{UltraCharSeq}. If $w_*$ is regularly varying of index $0<\alpha<1$, then there exists a constant $D\geq 1$ such that
\begin{equation} \label{weightRom}
	2 w_*(t) \leq w_* (Dt) + D
\end{equation}
holds for all $t\geq 0$. It is easy to check that \eqref{weightRom} fails for sufficiently large $t$ if $w_*$ is regularly varying of index $\alpha=0$, i.e., if $w_*$ is slowly varying, see Definitions \ref{SlowlyVF} and \ref{RegularlyVF}. However inequality \eqref{weightRom} is crucial in the sense that it holds if and only if there exists a constant $H\geq 1$ such that $M_{p+q} \leq H^{p+q} M_p M_q$ for all $p,q\in\N$, where $M_p$ is explained by \eqref{assocSeq}. For this we refer to Komatsu \cite[Proposition 3.6]{komatsu} and Pascu \cite[Proposition 3.1]{pascu}. We used this estimate in the proof of Proposition \ref{UltraCharSeq}. Moreover, if the assumptions of Proposition \ref{UltraCharSeq} are satisfied, then Corollary 16 in \cite{bonet} gives $\mathcal{E}_{\{w_*\}} (\R^n) = \mathcal{E}_{\{M_p\}} (\R^n)$. This makes clear that for $f\in \mathcal{M}^{w_*}_{p,q} (\R^n)$ the statements of Proposition \ref{UltraCharFunc} and Proposition \ref{UltraCharSeq} are equal if the weight function $w_* \in \mathcal{W}(\R)$ is regularly varying of index $\alpha \in (0,1)$. However, if $w_*$ is slowly varying, then $f\in \mathcal{E}_{\{w_*\}} (\R^n)$ but $f \notin \mathcal{E}_{\{M_p\}} (\R^n)$.
\end{rem}

\begin{ex}
	Let $w_*(|k|) = \langle k \rangle^{\frac{1}{s}}$ with $s>1$, that is $\mathcal{M}^{w_*}_{p,q} (\R^n) = \mathcal{GM}^s_{p,q}(\R^n)$. Then one can show that $w_* \in \mathcal{W}(\R)$ and $w_*$ is regularly varying of index $0< \frac{1}{s}<1$. Furthermore, the associated sequence of $w_*$ is $M_p = (p!)^s$, see \cite{bonet} and Rodino \cite[Proposition 1.4.2]{rodino}. Due to Proposition \ref{UltraCharSeq} there exists a constant $C$ such that
	\[ |D^\alpha f| \leq C^{|\alpha|+1} (\alpha !)^s \]
	for all $f\in \mathcal{GM}^s_{p,q}(\R^n)$. But this is exactly the definition of Gevrey spaces $G^s (\R^n)$ in the sense of Definition 1.4.1. in \cite{rodino}. Thus, we haveµ shown $\mathcal{GM}^s_{p,q}(\R^n) \subset G^s (\R^n)$. So applying the theory of ultradifferentiable functions in fact improved the estimate of Corollary 2.11 in \cite{rrs}.
\end{ex}

	\section{Multiplication Algebras} \label{MultAlg}
	
	In this section our goal is to prove that the modulation spaces $\mathcal{M}^{w_*}_{p,q} (\R^n)$ are algebras under pointwise multiplication if $w_* \in \mathcal{W}(\R)$. This property immediately explains analytic superposition due to Taylor expansion. \\
	The following lemma is a fundamental tool in order to prove later results. 
\begin{lma}\label{numerGeneral}
	Let $w_* \in \mathcal{W}(\R)$ be a fixed weight function. Suppose $\tilde{x}$ to be as in \eqref{Xstart} associated to $w_*$. Then there exists a positive real number $s$ such that
	\begin{equation}\label{ws-20-gen}
		w_* (x) \le w_* (y) + w_* (x-y) - s \, \min (w_* (y), w_* (x-y))
	\end{equation}
	holds for all $(x,y) \in \R_+^2 \setminus \{ (x,y) \in \R_+^2 : y\leq x < 2 \tilde{x}\}$.
\end{lma}
\begin{proof}
Let $y\geq x$. From $w_*$ increasing  and $\min (w_* (y), w_* (x-y)) = w_* (x-y)$ we derive the validity of \eqref{ws-20-gen} with $0 < s\le 1$. \\
Now we turn to the case $x>y$. Let $x \geq 2 \tilde{x}$. We divide the proof in two steps. \\
{\em Step 1.} Suppose $0\leq y\leq x\leq 2y$, i.e., $\min(w_*(y), w_*(x-y)) = w_*(x-y)$. Note that it immediately follows $\tilde{x} \leq \frac{x}{2} \leq y \leq x$. We consider the function 
\[ h(x,y) := w_*(x) - w_*(y) -(1-s) w_*(x-y). \]
We perform basic extreme value computations and obtain
 \begin{eqnarray*}
 	w'_*(x) - (1-s) w'_* (x-y) & = & 0, \\
	-w'_*(y) + (1-s) w'_*(x-y) & = & 0,
 \end{eqnarray*}
 which yields $w'_*(x) = w'_*(y)$ for all $x,y$ satisfying the assumptions. In particular, we have monotone behavior for $w'_*$ in the considered interval which yields $x=y$. Moreover, this gives $w'_*(x) = w'_*(y) = 0$ taking account of \ref{A2}. But this is a contradiction because of the monotone behavior of $w'_*$, cf. \ref{A3}. Thus, $h$ attains its maximum on the boundary of the described domain, that is, $x=2\tilde{x}$, $x=y$, $y=\frac{x}{2}$. Hence, we need to consider
 \begin{eqnarray}
 	h(2 \tilde{x}, y) & = & w_*(2\tilde{x}) - w_*(y) - (1-s) w_*(2 \tilde{x} -y) \mbox{ for } \tilde{x}\leq y \leq 2 \tilde{x}, \qquad\label{c1-1-1-gen} \\
	h(x,x) & = & -(1-s) w_* (0), \label{c1-1-2-gen} \\
	h(x,\frac{x}{2}) & = & w_*(x) - (2-s) w_*(\frac{x}{2}). \label{c1-1-3-gen}
\end{eqnarray}
The function \eqref{c1-1-2-gen} is trivially nonpositive for all $s\in (0,1]$ due to \ref{A5}, which implies \eqref{ws-20-gen}. Considering the function \eqref{c1-1-1-gen} and taking account of the mean value theorem we obtain
\[ w'_*(\xi) \frac{2\tilde{x} -y}{w_*(2\tilde{x}-y)} \leq 1-s, \]
where $\xi \in (y,2\tilde{x})$. A change of variables yields
\begin{equation} \label{Ineq1}
	\frac{z}{w_*(z)} w'_*(\xi) \leq 1-s
\end{equation}
with $0\leq z \leq \tilde{x}$ and $\xi \in (2\tilde{x} -z, 2 \tilde{x})$. By \ref{A6}, \ref{A7} and \ref{A8} it follows
\[ \frac{z}{w_*(z)} w'_*(\xi) \leq \frac{\tilde{x}}{w_*(\tilde{x})} w'_*(\xi) \leq \frac{\tilde{x}}{w_*(\tilde{x})} w'_*(\tilde{x}) \leq C <1. \]
Thus, \eqref{Ineq1} is true and we can put $s:= 1-C>0$. \\
It is left to show that the function \eqref{c1-1-3-gen} is nonpositive. Similar arguments as before give that $h(x, \frac{x}{2}) \leq 0$ if 
\[  \frac{\frac{x}{2}}{w_*(\frac{x}{2})} w'_*(\xi) \leq (1-s), \]
where $\xi\in (\frac{x}{2},x)$. Due to \ref{A6}, \ref{A7} and \ref{A8} there exists a positive constant $0<C<1$ such that 
\[  \frac{\frac{x}{2}}{w_*(\frac{x}{2})} w'_*(\xi) \leq \frac{\frac{x}{2}}{w_*(\frac{x}{2})} w'_*(\frac{x}{2}) \leq C \]
because $\frac{x}{2}\geq \tilde{x}$. Again we put $s:= 1-C>0$. \\
{\em Step 2.} Suppose $0\leq 2y\leq x$, i.e., $\min(w_*(y), w_*(x-y)) = w_*(y)$. Note that it immediately follows $0 \leq y \leq \frac{x}{2}$. Therefore we consider the function 
\[ h(x,y) := w_*(x) - w_*(x-y) -(1-s)w_*(y). \]
By basic extreme value computations we obtain
 \begin{eqnarray*}
 	w'_*(x) - w'_* (x-y) & = & 0, \\
	w'_*(x-y) - (1-s) w'_*(y) & = & 0,
 \end{eqnarray*}
 which yields $w'_*(x) = w'_*(x-y)$ for all $x,y$ satisfying the assumptions. Due to the assumptions we also know that $x-y\geq \frac{x}{2} \geq \tilde{x}$ and the monotone behavior for $w'_*$ gives $y=0$, see \ref{A6}. Taking account of \ref{A2} we obtain $w'_*(x)=0$ which is a contradiction to \ref{A3}. Thus, $h$ attains its maximum on the boundary of the described domain, that is, $x=2\tilde{x}$, $y=0$, $y=\frac{x}{2}$. Hence, we need to consider
 \begin{eqnarray}
 	h(2 \tilde{x}, y) & = & w_*(2\tilde{x}) - w_*(2\tilde{x} - y) - (1-s) w_*(y) \mbox{ for } 0 \leq y \leq \tilde{x}, \qquad\label{c1-2-1-gen} \\
	h(x,0) & = & -(1-s) w_* (0), \label{c1-2-2-gen} \\
	h(x,\frac{x}{2}) & = & w_*(x) - (2-s) w_*(\frac{x}{2}). \label{c1-2-3-gen}
\end{eqnarray}
The functions \eqref{c1-2-2-gen} and \eqref{c1-2-3-gen} were already considered in Step 1. It remains to show that the function \eqref{c1-2-1-gen} is nonpositive. Following the same arguments as in Step 1 we get $h(2 \tilde{x}, y) \leq 0$ if 
\[ \frac{y}{w_*(y)} w'_*(\xi) \leq 1-s, \]
where $y\in [0,\tilde{x}]$ and $\xi \in (2\tilde{x} -y, 2\tilde{x})$. Due to \ref{A6}, \ref{A7} and \ref{A8} it follows
\[ \frac{y}{w_*(y)} w'_*(\xi) \leq \frac{\tilde{x}}{w_*(\tilde{x})} w'_*(\xi) \leq \frac{\tilde{x}}{w_*(\tilde{x})} w'_*(\tilde{x}) \leq C <1 \]
and we put $s:= 1-C>0$. \\
The proof is complete.
\end{proof}

Based on Lemma \ref{numerGeneral} we are now able to formulate the main result of this section.

\begin{thm} \label{algebraUltraGen}
Let $1\leq p_1,p_2,q \leq \infty$ and $w_* \in \mathcal{W}(\R)$. Define $p$ by $\frac 1p := \frac{1}{p_1}+\frac{1}{p_2}$.
Assume that $f\in\mathcal{M}^{w_*}_{p_1,q}(\R^n)$ and $g\in\mathcal{M}^{w_*}_{p_2,q}(\R^n)$. Then $f\cdot g\in\mathcal{M}^{w_*}_{p,q}(\R^n)$ 
and it holds
\[ \|f \, g\|_{\mathcal{M}^{w_*}_{p,q}} \leq C \|f\|_{\mathcal{M}^{w_*}_{p_1,q}} \|g\|_{\mathcal{M}^{w_*}_{p_2,q}} \]
with a positive constant $C$ which only depends on the choice of the frequency-uniform decomposition, 
the dimension $n$, the parameter $q$ and the weight $w_*$.
\\
In particular, the modulation space $\mathcal{M}^{w_*}_{p,q}(\R^n)$ is an algebra under pointwise multiplication.
\end{thm}

\begin{proof}
We can basically follow the proof of Theorem 2.15 and Corollary 2.17 in \cite{rrs}. \\
{\em Step 1.} Subsequently we will use the notations $f_j (x)= \Box_j f(x)$ and $g_l (x)= \Box_l g(x)$ for $j,l\in\Z^n$.
A formal representation of the product $f\cdot g$ is given by 
\[ f\cdot g = \sum_{j,l\in \Z^n} f_j \cdot  g_l \, . \]
H\"older's inequality yields
\begin{eqnarray*}
\Big|\sum_{j,l\in \Z^n} f_j \cdot g_l\, \Big| & \le & \Big(\sum_{j\in \Z^n} \|\, f_j\, \|_{L^\infty}^2\Big)^{1/2} 
\Big(\sum_{l \in \Z^n} \| \, g_l\, \|_{L^\infty}^2 \Big)^{1/2} \\
	& \le & C \, \|\,f\, \|_{\mathcal{M}^{w_*}_{\infty,\infty}} \|\,g\, \|_{\mathcal{M}^{w_*}_{\infty,\infty}}
\end{eqnarray*}
with a constant $C$ independent of $f$ and $g$. This shows convergence of $\sum_{j,l\in \Z^n} f_j \cdot g_l$ in $\S'(\R^n)$. In view of Lemma \ref{basic} (iii) it will be sufficient to prove that the sequence $( \sum_{|j|,|l| < N} f_j \cdot g_l)_N$ is uniformly bounded in $\mathcal{M}^{w_*}_{p,q} (\R^n)$. \\
{\em Step 2.} Support properties and an application of Young's and H\"older's inequality lead to
\begin{eqnarray*}
 	\Big( \sum_{k\in\Z^n} e^{w_*(|k|) q} && \hspace{-0.7cm} \|\F^{-1} \big(\sigma_k \F (f\cdot g) \big)\|_{L^p}^q \Big)^{\frac{1}{q}} \nonumber \\
		& \lesssim & \max_{\substack{t \in \Z^n, \\ -3<t_i < 3, \\ i=1,\ldots, n}} \left( \sum_{k\in\Z^n} e^{w_*(|k|)q} \left[ \sum_{l\in\Z^n} \| f_{t-(l-k)} \|_{L^{p_1}} \| g_l \|_{L^{p_2}} \right]^q \right)^{\frac{1}{q}}. \hspace{1cm}
\end{eqnarray*}
Taking account of Lemma \ref{numerGeneral} we continue the estimate as follows:
\begin{eqnarray*}
& & \hspace{-1cm} \max_{\substack{t \in \Z^n, \\ -3<t_i < 3, \\ i=1,\ldots, n}} \bigg( \sum_{k\in\Z^n} e^{ w_*(|k|) q} \Big[ \sum_{l\in\Z^n} \|f_{t-(l-k)}\|_{L^{p_1}} \|g_l\|_{L^{p_2}} \Big]^q \bigg)^{\frac{1}{q}} \hspace{4.5cm}\\
	& \leq & \max_{\substack{t \in \Z^n, \\ -3<t_i < 3, \\ i=1,\ldots, n}} \bigg( \sum_{\substack{k\in \Z^n, \\ |k|< 2\tilde{x}}} \Big[ \sum_{l \in \Z^n} e^{w_*(|k|)} \|f_{t-(l-k)}\|_{L^{p_1}} \|g_l\|_{L^{p_2}} \Big]^q \\
	& & \qquad \qquad \qquad + \sum_{\substack{k\in \Z^n, \\ |k| \geq 2\tilde{x}}} \Big[ \sum_{l \in \Z^n} e^{w_*(|k|)} \|f_{t-(l-k)}\|_{L^{p_1}} \|g_l\|_{L^{p_2}} \Big]^q \bigg)^{\frac{1}{q}} 
\end{eqnarray*}
\begin{eqnarray*}
	& \leq &  e^{w_*(2\tilde{x})}  \, \max_{\substack{t \in \Z^n, \\ -3<t_i < 3, \\ i=1,\ldots, n}} \bigg( \sum_{\substack{k\in \Z^n, \\ |k|< 2\tilde{x}}} \Big[ \sum_{l \in \Z^n} \|f_{t-(l-k)}\|_{L^{p_1}} \|g_l\|_{L^{p_2}} \Big]^q \\
	& & \qquad \qquad \qquad + \sum_{\substack{k\in \Z^n, \\ |k| \geq 2\tilde{x}}} \Big[ \sum_{\substack{l \in \Z^n, \\ |l|\leq |l-k|}} e^{w_*(|l-k|)} \|f_{t-(l-k)}\|_{L^{p_1}} e^{w_*(|l|)} \|g_l\|_{L^{p_2}} e^{-s w_*(|l|)} \\
	& & \qquad \qquad \qquad + \sum_{\substack{l \in \Z^n, \\ |l-k|\leq |l|}} e^{w_*(|l-k|)} \|f_{t-(l-k)}\|_{L^{p_1}} e^{w_*(|l|)} \|g_l\|_{L^{p_2}} e^{-sw_*(|l-k|)} \Big]^q \bigg)^{\frac{1}{q}} \\
	& \lesssim & \max_{\substack{t \in \Z^n, \\ -3<t_i < 3, \\ i=1,\ldots, n}} \Bigg\{ \bigg(\sum_{k\in \Z^n} \Big[ \sum_{\substack{l \in \Z^n, \\ |l|\leq |l-k|}} e^{w_*(|l-k|)} \|f_{t-(l-k)}\|_{L^{p_1}} e^{w_*(|l|)} \|g_l\|_{L^{p_2}} e^{-s w_*(|l|)} \Big]^q \bigg)^{\frac{1}{q}} \\
	& & \qquad \qquad \qquad + \bigg( \sum_{k\in\Z^n} \Big[ \sum_{\substack{l \in \Z^n, \\ |l-k|\leq |l|}} e^{w_*(|l-k|)} \|f_{t-(l-k)}\|_{L^{p_1}} e^{w_*(|l|)} \|g_l\|_{L^{p_2}} e^{-sw_*(|l-k|)} \Big]^q \bigg)^{\frac{1}{q}} \Bigg\} \\
	& =: & \max_{\substack{t \in \Z^n, \\ -3<t_i < 3, \\ i=1,\ldots, n}} \big\{  S_{1,t} + S_{2,t} \big\}
\end{eqnarray*}
with $s$ and $\tilde{x}$ as in Lemma \ref{numerGeneral}. Note that \ref{A5} implies $w_* (|l-k|) + w_* (|l|) - s \, \min (w_* (|l|), w_* (|l-k|)) \geq 1$. A change of variables $j = l-k$ together with H\"older's inequality yields the following estimate:
 \begin{eqnarray*}
	S_{1,t} & = & \bigg(\sum_{k\in \Z^n} \Big[ \sum_{\substack{j \in \Z^n, \\ |j+k|\leq |j|}} e^{w_*(|j|)} \|f_{t-j}\|_{L^{p_1}} e^{w_*(|j+k|)} \|g_{j+k}\|_{L^{p_2}} e^{-s w_*(|j+k|)} \Big]^q \bigg)^{\frac{1}{q}} \\
		& \leq & \bigg(\sum_{k\in \Z^n} \sum_{\substack{j \in \Z^n, \\ |j+k|\leq |j|}} e^{q w_*(|j|)} \|f_{t-j}\|^q_{L^{p_1}} e^{q w_*(|j+k|)} \|g_{j+k}\|^q_{L^{p_2}} \Big( \sum_{\substack{j \in \Z^n, \\ |j+k|\leq |j|}} e^{-s q' w_*(|j+k|)} \Big)^{\frac{q}{q'}} \bigg)^{\frac{1}{q}} \\	
	& \lesssim & \bigg(\sum_{k\in \Z^n} \sum_{\substack{j \in \Z^n, \\ |j+k|\leq |j|}} e^{q w_*(|j|)} \|f_{t-j}\|^q_{L^{p_1}} e^{q w_*(|j+k|) } \|g_{j+k}\|^q_{L^{p_2}} \bigg)^{\frac{1}{q}} \\
		& \leq & \bigg( \sum_{j \in \Z^n} e^{q w_*(|j|)} \|f_{t-j}\|^q_{L^{p_1}} \sum_{k\in \Z^n} e^{q w_*(|j+k|)} \|g_{j+k}\|^q_{L^{p_2}} \bigg)^{\frac{1}{q}},
\end{eqnarray*}
where we used the fact 
	\begin{equation} \label{SeriesConv}
		\Big( \sum_{m \in\Z^n}  e^{-s q' w_* (|m|)} \Big)^{\frac{1}{q'}} <\infty 
	\end{equation}
due to \ref{A4} and $\frac{1}{q} + \frac{1}{q'} = 1$. By using the mean-value theorem, the uniform boundedness and the positivity of $w_*'$, see \ref{A1}, \ref{A2}, \ref{A6} and \ref{A3}, we get
\[ e^{w_* (|j|) - w_* (|t-j|)} = e^{w_*' (\xi)(|j|-|j-t|)}\le e^{\|\, w_*' \|_{L^\infty}\, |t|} \]
for some $\xi \in \R$. Hence,
	\[ \max_{\substack{t \in \Z^n, \\ -3<t_i < 3, \\ i=1,\ldots, n}} \sup_{j\in\Z^n} e^{w_* (|j|) - w_* (|t-j|)} \leq e^{\|w'_*\|_{L^\infty} \, 2\sqrt{n}} < \infty \, . \]
Thus, we obtain
\[ \max_{\substack{t \in \Z^n, \\ -3<t_i < 3, \\ i=1,\ldots, n}} S_{1,t} \lesssim \| f\|_{\mathcal{M}^{w_*}_{p_1,q}} \| g \|_{\mathcal{M}^{w_*}_{p_2,q}}  \]
and analogously
\[ \max_{\substack{t \in \Z^n, \\ -3<t_i < 3, \\ i=1,\ldots, n}} S_{2,t} \lesssim \| f\|_{\mathcal{M}^{w_*}_{p_1,q}}  \| g\|_{\mathcal{M}^{w_*}_{p_2,q}}. \]
Note that all constants are independent of $f$ and $g$. \\
{\em Step 3.} By applying Step 2 with $p_1=p_2 =2p$  and stressing Lemma \ref{einbettung} we prove the algebra property of $\mathcal{M}^{w_*}_{p,q}(\R^n)$, which completes the proof.
\end{proof}
	\section{Non-analytic Superposition Operators} \label{NonAnalSup}	
	
We adopt the strategy of \cite{rrs} in order to obtain a non-analytic superposition result on the modulation spaces $\mathcal{M}^{w_*}_{p,q}(\R^n)$. Note that this strategy got introduced in \cite{brs}. \\
Propositions \ref{UltraCharFunc} and \ref{UltraCharSeq} have already shown that we can distinguish two subclasses of functions belonging to the class $\mathcal{W}(\R)$. Therefore we divide $\mathcal{W}(\R)$ into:
\begin{eqnarray*}
	\mathcal{W}_1 (\R) & := & \{ f \in \mathcal{W} (\R): f \text{ is regularly varying of index } \alpha \in (0,1) \} \, , \\
	\mathcal{W}_0 (\R) & := & \{ f \in \mathcal{W} (\R): f \text{ is slowly varying } \} \, .
\end{eqnarray*}
Obviously it holds $\mathcal{W}(\R) = \mathcal{W}_0 (\R) \cup \mathcal{W}_1 (\R)$. Remark that if $w_* \in \mathcal{W}_1 (\R)$, then there exists a slowly varying function $\widetilde{w}_*$ such that 
\begin{equation} \label{ws-51}
	w_* (x) = x^\alpha \widetilde{w}_*(x) 
\end{equation}
due to the Characterisation Theorem, see \cite[Theorem 1.4.1]{bgt}. Since $w_* \in C^\infty (0, \infty)$, see \ref{A1}, it follows that $\widetilde{w}_* \in C^\infty (0,\infty)$. Hence, either $\widetilde{w}_* (t)$ is bounded for all $t \in (0,\infty)$ or tends to infinity as $t \to \infty$. Note that because of \ref{A5} and \ref{A3} we only consider positive weight functions $w_*$. \\
At first we show the subalgebra property for modulation spaces $\mathcal{M}^{w_*}_{p,q} (\R^n)$ which is an important tool to prove the main result of this section. Therefore we decompose the phase space in $(2^n+1)$ parts. Let $R>0$ and $\epsilon=(\epsilon_1,\ldots, \epsilon_n)$ be fixed with $\epsilon_j \in \{0,1\}$, $j=1,\ldots,n$. Then we put
\[ P_R := \{ \xi\in\R^n :\: |\xi_j|\leq R, \: j=1,\ldots, n \} \]
and
\[ P_R(\epsilon) := \{ \xi\in \R^n: \: \sgn (\xi_j) = (-1)^{\epsilon_j}, \: j=1,\ldots,n \} \setminus P_R. \]

\subsection{The Class $\mathcal{W}_1 (\R)$ of Weight Functions}

\begin{prop} \label{subalgebraUltraGenRV}
Let $1\leq p, q \leq \infty$.  Suppose that $\epsilon=(\epsilon_1,\ldots, \epsilon_n)$ is fixed with 
$\epsilon_j \in \{0,1\}$, $j=1,\ldots,n$. Let $w_*\in \mathcal{W}_1 (\R)$ be the fixed weight with the representation \eqref{ws-51} and let $R\geq 2$.
The spaces
\[ \mathcal{M}^{w_*}_{p,q}(\epsilon, R) := \{ f\in \mathcal{M}^{w_*}_{p,q}(\R^n): \supp \F(f) \subset P_R(\epsilon) \} \]
are subalgebras of $\mathcal{M}^{w_*}_{p,q} (\R^n)$. Furthermore, it holds
\begin{equation}\label{ws-29-gen}
\|f \, g\|_{\mathcal{M}^{w_*}_{p,q}} \leq D_R \, \|f\|_{\mathcal{M}^{w_*}_{p,q}} \|g\|_{\mathcal{M}^{w_*}_{p,q}} 
\end{equation}
for all $f,g \in\mathcal{M}^{w_*}_{p,q}(\epsilon, R)$. The constant $D_R$ can be specified in the following way: \\
\begin{itemize}
\item[(a)] Suppose $\liminf_{t\to \infty} \widetilde{w}_* (t) = C_0 >0$, where $C_0 = \infty$ is allowed. Then $D_R$ is given by
	\[ D_R := C \bigg( \int_{ sq' c (R-2)^{\alpha}}^\infty y^{n/\alpha-1} e^{-y} \, dy \bigg)^{\frac{1}{q'}}, \]
where the constant $C > 0 $ depends only on $n,p,q$ and  $w_*$, but not on $f,g$ and $R$.  Moreover, $c := \min(1,C_0)$. \\
\item[(b)]  Let $\delta >0$ be such that $\alpha - \delta >0$. If $\liminf_{t\to \infty} \widetilde{w}_* (t) = 0$, then $D_R$ is given by
	\[ D_R := C \bigg( \int_{s q'  (R-2)^{\alpha-\delta}}^\infty y^{n/\alpha -1} e^{-y} \, dy \bigg)^{\frac{1}{q'}} \, , \]
where the constant $C > 0 $ depends only on $n,p,q$ and  $w_*$, but not on $f,g$ and $R$.
\end{itemize}
\end{prop}

\begin{proof}
For details see the proof of Proposition 3.1 in \cite{rrs}. Moreover, taking account of the Fourier supports of the functions $f$ and $g$ we proceed as in the proof of Theorem \ref{algebraUltraGen}. Let
 \[ P_R^* (\epsilon) := \{ k\in \Z^n : \max_{j=1,\ldots, n} | k_j | > R-1, \quad \sgn (k_j) = (-1)^{\epsilon_j} , j=1,\ldots, n \}. \]
Then we obtain
\begin{eqnarray*}
	 \Big( \sum_{k\in P_R^* (\epsilon) } e^{w_*(|k|) q} && \hspace{-0.7cm} \|\F^{-1} \big(\sigma_k \F (f\cdot g) \big)\|_{L^p}^q \Big)^{\frac{1}{q}} \nonumber \\
		& \lesssim & \max_{\substack{t \in \Z^n, \\ -3<t_i < 3, \\ i=1,\ldots, n}} \Bigg( \sum_{k\in P_R^* (\epsilon) } e^{w_*(|k|)q} \Big[ \sum_{\substack{l\in\Z^n, \\ l, t-(l-k) \in P_R^* (\epsilon) }} \| f_{t-(l-k)} \|_{L^{p_1}} \| g_l \|_{L^{p_2}} \Big]^q \Bigg)^{\frac{1}{q}} \\
		& \lesssim &  \max_{\substack{t \in \Z^n, \\ -3<t_i < 3, \\ i=1,\ldots, n}} \big\{  S_{1,t} + S_{2,t} \big\},
\end{eqnarray*}
where 
\begin{eqnarray*}
	S_{1,t} & \leq & \Bigg(\sum_{k\in P_R^* (\epsilon)} \sum_{\substack{ l \in \Z^n, \\ l, t-(l-k) \in P_R^* (\epsilon), \\ |l|\leq |l-k|}} e^{q w_*(|l-k|)} \|f_{t-(l-k)}\|^q_{L^{p_1}} e^{q w_*(|l|)} \|g_{l}\|^q_{L^{p_2}} \\
		& & \qquad \qquad \qquad \times \, \bigg( \sum_{\substack{l \in \Z^n, \\ l, t-(l-k) \in P_R^* (\epsilon), \\ |l|\leq |l-k|}} e^{-s q' w_*(|l|)} \bigg)^{\frac{q}{q'}} \Bigg)^{\frac{1}{q}}, \\
	S_{2,t} & \leq & \Bigg(\sum_{k\in P_R^* (\epsilon)} \sum_{\substack{ l \in \Z^n, \\ l, t-(l-k) \in P_R^* (\epsilon), \\ |l-k|\leq |l|}} e^{q w_*(|l-k|)} \|f_{t-(l-k)}\|^q_{L^{p_1}} e^{q w_*(|l|)} \|g_{l}\|^q_{L^{p_2}} \\
		& & \qquad \qquad \qquad \times \, \bigg( \sum_{\substack{l \in \Z^n, \\ l, t-(l-k) \in P_R^* (\epsilon), \\ |l-k|\leq |l|}} e^{-s q' w_*(|l-k|)} \bigg)^{\frac{q}{q'}} \Bigg)^{\frac{1}{q}}.
\end{eqnarray*}
In the following estimates let us consider the set $P_R^* (\epsilon)$ for $\epsilon = (0, \ldots, 0)$. The other cases $\epsilon \neq (0, \ldots, 0)$ can be deduced from a symmetry argument. Now we distinguish two cases depending on $\widetilde{w}_*$. \\
{\em Step 1.} Suppose $\liminf_{t\to \infty} \widetilde{w}_* (t) =: C_0 >0$, where $C_0 = \infty$ is allowed. With $R\geq 2$ and employing equality \eqref{ws-51} we estimate 
\begin{eqnarray*}
	\sum_{\substack{l \in \Z^n, \\ l, t-(l-k) \in P_R^* (0,\ldots, 0), \\ |l|\leq |l-k|}} e^{-s q' w_*(|l|)} & \leq & \sum_{\substack{ l\ \in\Z^n, \\ \max_{j=1,\ldots,n} |l_j| > R-1}}  e^{-s q' w_*(|l|)} \\
		& \leq & \int_{|x|>R-2} e^{-sq' |x|^\alpha \widetilde{w}_*(|x|)} \, dx \hspace{4cm}
\end{eqnarray*}
\begin{eqnarray*}
		& \leq &  2\, \frac{\pi^{n/2}}{\Gamma (n/2)}\,  \int_{R-2}^\infty r^{n-1} e^{-s q' c r^{\alpha}} \, dr \\
		& \stackrel{[t=r^\alpha]}{=} & 2\, \frac{\pi^{n/2}}{\Gamma (n/2)}\, \frac{1}{\alpha} \, \int_{(R-2)^{\alpha}}^\infty t^{n/\alpha -1} e^{-s q' c t} \, dt \\
		& \stackrel{[y=s q' c t]}{=} & 2\, \frac{\pi^{n/2}}{\Gamma (n/2)}\, \frac{1}{\alpha}\,  (s q' c)^{-n/\alpha} \int_{s q'  c (R-2)^{\alpha}}^\infty y^{n/\alpha -1} e^{-y} \, dy \\ 
		& = : & E_R\, ,
\end{eqnarray*}
where $c := \min(1, C_0)$. \\
{\em Step 2.} Suppose $\liminf_{t\to \infty} \widetilde{w}_* (t) = 0$. With $R\geq 2$, $\delta>0$ arbitrarily small and employing equality \eqref{ws-51} we estimate 
\begin{eqnarray*}
	\sum_{\substack{l \in \Z^n, \\ l, t-(l-k) \in P_R^* (0,\ldots, 0), \\ |l|\leq |l-k|}} e^{-s q' w_*(|l|)} & \leq & \sum_{\substack{ l\ \in\Z^n, \\ \max_{j=1,\ldots,n} |l_j| > R-1}}  e^{-s q' w_*(|l|)} \\
		& \leq & \int_{|x|>R-2} e^{-sq' |x|^\alpha \widetilde{w}_*(|x|)} \, dx \\
		& \leq &  2\, \frac{\pi^{n/2}}{\Gamma (n/2)}\,  \int_{R-2}^\infty r^{n-1} e^{-s q' r^{\alpha-\delta}} \, dr \\
		& \stackrel{[t=r^{\alpha-\delta}]}{=} & 2\, \frac{\pi^{n/2}}{\Gamma (n/2)}\, \frac{1}{\alpha-\delta} \, \int_{(R-2)^{\alpha-\delta}}^\infty t^{n/ (\alpha-\delta) -1} e^{-s q' t} \, dt \\
		& \stackrel{[y=s q' t]}{=} & 2\, \frac{\pi^{n/2}}{\Gamma (n/2)}\, \frac{1}{\alpha-\delta}\,  (s q')^{-n/ (\alpha-\delta)} \int_{s q'  (R-2)^{\alpha-\delta}}^\infty y^{n/\alpha -1} e^{-y} \, dy \\ 
		& = : & F_R\, .
\end{eqnarray*}
In the last part of the proof the constant $D_R$ stands for $E_R$ and $F_R$, respectively. We follow the estimates used in the proof of Theorem \ref{algebraUltraGen} and obtain
\begin{eqnarray*}
	S_{1,t} & \leq & D_R^{1/q'} \Bigg(\sum_{j\in \Z^n} e^{q w_*(|j|)} \|f_{t-j}\|^q_{L^{p_1}} \sum_{k\in \Z^n} e^{q w_*(|j+k|)} \|g_{j+k}\|^q_{L^{p_2}} \Bigg)^{\frac{1}{q}} \\
		& \lesssim & D_R^{1/q'} \| f\|_{\mathcal{M}^{w_*}_{p_1,q}}  \| g\|_{\mathcal{M}^{w_*}_{p_2,q}}
\end{eqnarray*}
as well as 
\[ S_{2,t} \lesssim D_R^{1/q'} \| f\|_{\mathcal{M}^{w_*}_{p_1,q}}  \| g\|_{\mathcal{M}^{w_*}_{p_2,q}}. \]
Lemma \ref{einbettung} yields the desired result. 
\end{proof}
\begin{rem}
	An obvious application of Proposition \ref{subalgebraUltraGenRV}(a) is the subalgebra property of Gevrey-modulation spaces $\mathcal{GM}^s_{p,q} (\R^n)$, see Proposition 3.1 in \cite{rrs}. In this case we have the weight function $w_* (|k|) = \langle k \rangle^{\frac{1}{s}}$, $s>1$. Hence $\widetilde{w}_* (|k|) = 1$ in \eqref{ws-51}. \\
	Furthermore, Proposition \ref{subalgebraUltraGenRV}(a) shows that the spaces $\mathcal{M}^{w_*}_{p,q} (\R^n)$ are subalgebras under pointwise multiplication if $w_* (|k|) = \langle k \rangle^{\frac{1}{s}} l_m \langle k \rangle_*$, $s>1, m\in \N$, where $*$ stands for a sufficiently large number (depending on $m$) such that $l_m \langle k \rangle_* \geq 1$ for all $k\in \Z^n$. Here we have $\widetilde{w}_* (|k|) = l_m \langle k \rangle_*$ in \eqref{ws-51}. \\
	Let us consider the weight function $w_* (|k|) = \langle k \rangle^{\frac{1}{s}} \big(\log\langle k \rangle_e \big)^r$, $s>1$, $r<0$. Hence $\widetilde{w}_* (|k|) = \big(\log\langle k \rangle_e \big)^r$ in \eqref{ws-51}. Then Proposition \ref{subalgebraUltraGenRV}(b) is applicable for the space $\mathcal{M}^{w_*}_{p,q} (\R^n)$.
\end{rem}

Subsequently we assume every function to be real-valued. Let us mention a technical lemma stated in \cite{brs}. 
	\begin{lma} \label{lma45brs}
		Let $\beta>0$. Define 
		\[ f(t) := \int_{t}^\infty e^{-y} y^{\beta-1} \, dy, \qquad t\geq 0. \]
		The inverse $g$ of the function $f$ maps $(0,\Gamma(\beta)]$ onto $[0,\infty)$ and it holds
		\[ \lim_{u\downarrow 0} \frac{g(u)}{\log \frac{1}{u}} = 1. \]
	\end{lma}
	\begin{proof}
		Cf. Lemma 4.5. in \cite{brs}.
	\end{proof}
Now we formulate the following lemma which is the main tool for Theorem \ref{SuperpositionUltraGenRV}. For brevity we write $\| \cdot \|$ instead of $\| \cdot \|_{\mathcal{M}^{w_*}_{p,q}}$.

\begin{lma} \label{estSuperUltraGenRV}
Let $1<p <\infty$ and  $1\le q \le \infty$. Suppose $w_* \in \mathcal{W}_1 (\R)$ with the representation \eqref{ws-51} and $u \in \mathcal{M}^{w_*}_{p,q} (\R^n)$. 
\begin{itemize}
\item[(a)] If $\widetilde{w}_*$ is bounded, then it holds
\[ \| e^{iu}-1\| \leq c \, \|u\| \, 
\left\{
\begin{array}{lll}
 e^{ b \|u \|^\alpha \log \|u \| } &\qquad & \mbox{ if } \|u\| > 1, \\
1 &&  \mbox{ if } \|u\| \leq 1
\end{array}\right. 
\]
with constants $b,c >0$ independent of $u$.
\item[(b)] Suppose $\lim_{t\to \infty} \widetilde{w}_*(t) = \infty$. Then it holds
\[ \| e^{iu}-1\| \leq c \, \|u\| \, 
\left\{
\begin{array}{lll}
 e^{ b \|u \|^\alpha \log (\|u \|) \, \widetilde{w}_* ( a \| u \| (\log \|u \|)^{\frac{1}{\alpha}} ) } &\qquad & \mbox{ if } \|u\| > 1, \\
1 &&  \mbox{ if } \|u\| \leq 1
\end{array}\right. 
\]
with constants $a,b,c >0$ independent of $u$.
\end{itemize}
\end{lma}
\begin{proof}
We follow the proof of Lemma 3.6 in \cite{rrs}. Note that the idea of this proof originates from \cite{brs}. We give a comment to the estimate of the following term 
\begin{eqnarray} \label{ws-53}
	\sum_{\substack{k\in\Z^n, \\ -Rr-1<k_i< Rr+1, \\ i=1,\ldots,n}} e^{w_*(|k|) q} & \leq & \int_{|x|< \sqrt{n}(Rr+1)}\, e^{|x|^\alpha \widetilde{w}_* (|k|) q}  \, dx \nonumber \\
		& = & 2\, \frac{\pi^{n/2}}{\Gamma (n/2)}\,  \int_{0}^{\sqrt{n} (Rr+1)} \tau^{n-1} e^{ \tau^{\alpha}  \widetilde{w}_* (\tau) q} \, d\tau \, .
\end{eqnarray}
Now we distinguish two cases depending on $\widetilde{w}_*$, which comes from \eqref{ws-51}. \\
{\em Step 1.1.} Assume $\widetilde{w}_*$ is a bounded function. Then we put $c_0 := \sup_{t\in (0,\infty)} \widetilde{w}_*(t)$ and continue the estimate \eqref{ws-53} as follows:
\begin{eqnarray*}
	2\, \frac{\pi^{n/2}}{\Gamma (n/2)}\,  \int_{0}^{\sqrt{n} (Rr+1)} \tau^{n-1} e^{ \tau^{\alpha}  \widetilde{w}_* (\tau) q} \, d\tau & \leq & 2\, \frac{\pi^{n/2}}{\Gamma (n/2)}\,  \int_{0}^{\sqrt{n} (Rr+1)} \tau^{n-1} e^{ \tau^{\alpha} c_0 q} \, d\tau \\
	& \leq & C (\sqrt{n}(Rr + 1))^{n} e^{(\sqrt{n}(Rr+1))^\alpha c_0 q} \\
	& \leq & C e^{c_1 R^\alpha r^\alpha }
\end{eqnarray*}
with constants $C>0$ depending on $n$ and $c_1 >0$ depending on $n, q, \alpha, c_0$. \\ 
{\em Step 1.2.} The constant $D_R$ in Proposition \ref{subalgebraUltraGenRV} is strictly increasing and positive as a function of $R$ and we have $\lim_{R\to \infty} D_R = 0$. Following the proof of Lemma 3.6 in \cite{rrs} and setting
\[ \frac{D_R}{D_2} = \| u\|_{\mathcal{M}^{w_*}_{p,q}}^{\alpha -1} \]
together with Lemma \ref{lma45brs} yields the desired result.\\
{\em Step 2.1.} Suppose $\lim_{t\to \infty} \widetilde{w}_*(t) = \infty$. Then we continue the estimate \eqref{ws-53} as follows:
\begin{eqnarray*}
	2\, \frac{\pi^{n/2}}{\Gamma (n/2)}\,  \int_{0}^{\sqrt{n} (Rr+1)} \tau^{n-1} e^{ \tau^{\alpha}  \widetilde{w}_* (\tau) q} \, d\tau & \leq &C (\sqrt{n}(Rr + 1))^{n} e^{b (\sqrt{n}(Rr+1))^\alpha \widetilde{w}_* (\sqrt{n} (Rr +1)) q} \\
	& \leq & C e^{b c_1 R^\alpha r^\alpha \widetilde{w}_* (c_2 R r) }
\end{eqnarray*}
with constants $C>0$ depending on $n$, $b>0$ depending on $\widetilde{w}_*$, $c_1>0$ depending on $n, q, \alpha$ and $c_2 >0$ depending on $n$. \\
{\em Step 2.2.} The subalgebra constant $D_R$ is explained by Proposition \ref{subalgebraUltraGenRV}(a). Using the same arguments as in Step 1.2 gives the desired estimate. \\
The proof is complete.
\end{proof}

\begin{rem}
	The restriction with respect to $p$, i.e. $1<p<\infty$, appears in the proof of Lemma 3.6 in \cite{rrs} which we basically followed to prove Lemma \ref{estSuperUltraGenRV}. There the authors decompose a general function $u$ into $2^{n+1}$ parts by letting each part be Fourier supported on one of the sets $P_R (\epsilon)$ and $P_R$, respectively. This decomposition is realised by characteristic functions. For $p\in (1,\infty)$ characteristic functions on cubes are Fourier multipliers in $L^p (\R^n)$ by the Riesz Theorem. For further details we refer to Lizorkin \cite{lizorkin}.  
\end{rem}

\begin{lma} \label{ContExpUltraGen}
Let $1<p <\infty$ and  $1\le q \le \infty$. Let $w_* \in \mathcal{W}(\R)$. Assume $u\in \mathcal{M}^{w_*}_{p,q}(\R^n)$ to be fixed and define a function $g: \R \mapsto \mathcal{M}^{w_*}_{p,q}(\R^n)$ by 
$g(\xi) = e^{\i u(x) \xi}-1$. Then the function $g$ is continuous.
\end{lma}

\begin{proof}
We make us of the identity 
\[ e^{iu} -e^{iv} = (e^{iv} -1)(e^{i(u-v)} -1) + (e^{i(u-v)} -1)  \]
and the algebra property, see Theorem \ref{algebraUltraGen}. Due to Lemma \ref{estSuperUltraGenRV} the claim follows for $w_* \in \mathcal{W}_1 (\R)$. If $w_* \in \mathcal{W}_0 (\R)$, then Lemma \ref{estSuperUltraGenSV} below yields the claim. 
\end{proof}

Now we have all tools in order to state the main theorem of this section. 

\begin{thm} \label{SuperpositionUltraGenRV}
Let $1<p <\infty$ and $1\le q \le \infty$. Assume that $w_* \in \mathcal{W}_1 (\R)$ with the representation \eqref{ws-51}.
\begin{itemize}
\item[(a)] If $\widetilde{w}_*$ is bounded, then let $\mu$ be a complex measure on $\R$ such that
\begin{equation} \label{FourierEstUltraGenRV1}
	L_1(\lambda) := \int_{\R} e^{\lambda |\xi|^\alpha \log |\xi| } \, d|\mu| (\xi) < \infty
\end{equation}
for any $\lambda >0 $ and such that $\mu(\R) = 0$.

\item[(b)] If $\lim_{t\to\infty} \widetilde{w}_* (t) = \infty$, then assume $\mu$ to be a complex measure on $\R$ such that
\begin{equation} \label{FourierEstUltraGenRV2}
	L_2(\lambda) := \int_{\R} e^{ \lambda  |\xi|^\alpha \log (|\xi|)  \widetilde{w}_* ( |\xi| (\log |\xi|)^{\frac{1}{\alpha}}) } \, d|\mu| (\xi) < \infty
\end{equation}
for any $\lambda >0 $ and such that $\mu(\R) = 0$.
\end{itemize}
Let the function $f$ be the inverse Fourier transform of $\mu$. Then $f\in C^\infty (\R)$ and the composition operator $T_f: u \mapsto f \circ u$ maps $\mathcal{M}^{w_*}_{p,q} (\R^n)$ into $\mathcal{M}^{w_*}_{p,q} (\R^n)$.
\end{thm}
\begin{proof}
	Taking account of $\mu (\R) = 0$ the composition operator $T_f$ is given by
	\[ T_f u = (f \circ u) (x) = \frac{1}{\sqrt{2\pi}} \int_{\R} \big( e^{i u(x) \xi} - 1 \big) \, d\mu(\xi) \]
	for all $u \in \mathcal{M}^{w_*}_{p,q} (\R^n)$. Now we exactly follow the proof of Theorem 3.9 in \cite{rrs} by using equation \eqref{FourierEstUltraGenRV1} or respectively \eqref{FourierEstUltraGenRV2}, Lemma \ref{ContExpUltraGen} and Lemma \ref{estSuperUltraGenRV}. Note that in case $(b)$ we additionally use the fact that for all positive fixed numbers $d$ there exists a constant $C$ such that 
	\[ \widetilde{w}_* (dt) \leq C \widetilde{w}_*(t) \]
	for all $t>0$ due to the assumptions. For sufficiently large $t$ we can even put $C:=1+\delta$ for any $\delta>0$ since $\widetilde{w}_*$ is slowly varying. This completes the proof.
\end{proof}

Note the following conclusion.

\begin{cor} \label{SuperpositionCorUltraGenRV}
	Let $1<p <\infty$, $1\le q \le \infty$ and $w_* \in \mathcal{W}_1 (\R)$ with the representation \eqref{ws-51}. Let $\mu$ be a complex measure on $\R$ with the corresponding bounded density function $g$, i.e., $d\mu (\xi) = g(\xi) \, d\xi$. 
	\begin{itemize}
	\item[(a)] If $\widetilde{w}_*$ is bounded, then assume
		\begin{equation*} 
			\lim_{|\xi|\to \infty} \frac{|\xi|^\alpha \log|\xi|}{\log |g(\xi)|} = 0
		\end{equation*}
		and $\displaystyle \int_{\R} d\mu (\xi) = \int_{\R} g(\xi) \, d\xi = 0$.
	\item[(b)] If $\lim_{t\to\infty} \widetilde{w}_* (t) = \infty$, then assume
		\begin{equation*} 
			\lim_{|\xi|\to \infty} \frac{|\xi|^\alpha \log|\xi| \widetilde{w}_* (|\xi| (\log |\xi| )^{\frac{1}{\alpha}})}{\log |g(\xi)|} = 0
		\end{equation*}
		and $\displaystyle \int_{\R} d\mu (\xi) = \int_{\R} g(\xi) \, d\xi = 0$.
	\end{itemize}
Let $f$ be the inverse Fourier transform of $g$. Then $f\in C^\infty (\R)$ and the composition operator $T_f: u \mapsto f \circ u$ maps $\mathcal{M}^{w_*}_{p,q} (\R^n)$ into $\mathcal{M}^{w_*}_{p,q} (\R^n)$. 
\end{cor}
\begin{proof}
We follow the proof of  Corollary 3.11 in \cite{rrs} which yields the desired result.
\end{proof}

\begin{ex}
In Section \ref{WeightClass} the general weight function \eqref{WeightEx1} was introduced as a natural extension of the weights discussed in \cite{rrs}. If $s\in (1,\infty)$ and $r_k \in \R$, $k=1,2,\ldots, m$, then $w_* \in \mathcal{W}_1 (\R)$ is a regularly varying function of index $\frac{1}{s}$. For particular examples let us consider again the weight functions mentioned in the remark of Proposition \ref{subalgebraUltraGenRV}. If $w_* (|k|) = \langle k \rangle^{\frac{1}{s}} l_m \langle k \rangle_*$, $s>1, m\in \N$, then Theorem \ref{SuperpositionUltraGenRV}(b) yields a non-analytic superposition operator on $\mathcal{M}^{w_*}_{p,q} (\R^n)$. Given the weight function $w_* (|k|) = \langle k \rangle^{\frac{1}{s}} \big(\log\langle k \rangle_e \big)^r$, $s>1$, $r<0$, Theorem \ref{subalgebraUltraGenRV}(a) gives the desired non-analytic superposition in the modulation space $\mathcal{M}^{w_*}_{p,q} (\R^n)$. Note that for $s=1$ in \eqref{WeightEx1} condition \ref{A1} in Definition \ref{DefClassWeight} is violated and our methods are not applicable. \\
Now we recall one example of a particular superposition operator on $\mathcal{M}^{w_*}_{p,q} (\R^n)$ for the case $w_*(|k|) = \langle k \rangle^{\frac{1}{s}}$, $s>1$, which was already treated in \cite{rrs}. Let us consider the construction in Rodino \cite[Example 1.4.9]{rodino}. Let $\mu <0$ and let for $t \in \R$
\[ \psi_\mu (t):= \left\{ \begin{array}{lll} 
	e^{-t^\mu} & \qquad & \mbox{if} \quad t>0\, , \\
		0 && \mbox{otherwise}\, .
	\end{array} \right. \]
By taking
\[ \varphi_\mu (t):= \psi_\mu (1-t)\, \cdot \psi_\mu (t)\, , \qquad t\in \R\, , \]
we obtain a compactly supported $C^\infty$ function on $\R$. It follows $\varphi_\mu \in G^s (\R)$ for $s= 1-1/\mu$. Here the classes $G^s(\R)$ refer to classical Gevrey regularity, see \cite[Definition 1.4.1]{rodino}.
We skip the definition here and recall a further result, see \cite[Theorem 1.6.1]{rodino}. Since $\varphi_\mu \in G^s (\R)$  has compact support there exists a positive constant $c$ and some $\varepsilon >0$ such that
\[ |\F \varphi_\mu (\xi)| \le c \, e^{-\varepsilon \, |\xi|^{1/s}}\, , \quad \xi \in \R\, . \]
Because of $\varphi_\mu (0) =0$ Corollary \ref{SuperpositionCorUltraGenRV}(a) yields the following.
\end{ex}

\begin{cor} \label{ex1}
	Let the weight parameter $s>1$, $1 <p < \infty$ and  $1 \leq q \leq \infty$. Put $w_*(|k|) = \langle k \rangle^{\frac{1}{s}}$. Suppose that $\mu < \frac 1{1-s}$. Then  $T_{\varphi_\mu}: u \mapsto \varphi_\mu \circ u$ maps $\mathcal{M}^{w_*}_{p,q} (\R^n)$ into $\mathcal{M}^{w_*}_{p,q} (\R^n)$.
\end{cor}

\subsection{The Class $\mathcal{W}_0 (\R)$ of Weight Functions}

At first we state the subalgebra property for slowly varying weight functions, i.e., the counterpart of Proposition \ref{subalgebraUltraGenRV}. For this we recall the decomposition of the phase space $\R^n$ introduced above.

\begin{prop} \label{subalgebraUltraGenSV}
Let $1\leq p, q \leq \infty$ and $N \in \N$.  Suppose that $\epsilon=(\epsilon_1,\ldots, \epsilon_n)$ is fixed with 
$\epsilon_j \in \{0,1\}$, $j=1,\ldots,n$. Let $w_*\in \mathcal{W}_0 (\R)$ be the fixed weight and $R\geq 2$.
The spaces
\[ \mathcal{M}^{w_*}_{p,q}(\epsilon, R) := \{ f\in \mathcal{M}^{w_*}_{p,q}(\R^n): \supp \F(f) \subset P_R(\epsilon) \} \]
are subalgebras of $\mathcal{M}^{w_*}_{p,q} (\R^n)$. Furthermore, it holds
\begin{equation}\label{ws-30-gen}
\|f \, g\|_{\mathcal{M}^{w_*}_{p,q}} \leq G_{R,N} \, \|f\|_{\mathcal{M}^{w_*}_{p,q}} \|g\|_{\mathcal{M}^{w_*}_{p,q}} 
\end{equation}
	for all $f,g \in\mathcal{M}^{w_*}_{p,q}(\epsilon, R)$. The constant $G_{R,N}$ can be specified by
	\[ G_{R,N} := C_N \, R^{-N}, \]
	where the constant $C_N > 0 $ depends  on $n,p,q$ and  $N$, but not on $f,g$ and $R$.
\end{prop}
\begin{proof}
We proceed as in the proof of Proposition \ref{subalgebraUltraGenRV} and use the same notation. In particular, the quantities $S_{1,t}$, $S_{2,t}$ and the set $P_R^* (\epsilon)$ are taken over. Note that the proof strategy comes from Proposition 4.7 in \cite{rrs}. \\
Due to property \ref{A4} of $w_*$ there exists a constant $c_M$
such that 
\[
\sup_{l\in\Z^n} e^{-s\,  w_* (|l|)} \, (1+|l|)^M =:c_M < \infty
\]
for any natural number $M$. Note that it is sufficient to consider the case $\epsilon = (0, \ldots, 0)$ for $P_R^* (\epsilon)$. 
In order to estimate $S_{1,t}$ we consequently need to estimate again the quantity
\begin{eqnarray*}
	\sum_{\substack{l \in \Z^n, \\ l, t-(l-k) \in P_R^* (0,\ldots, 0), \\ |l|\leq |l-k|}} e^{-s q' w_*(|l|)} & \leq & \sum_{\substack{ l\ \in\Z^n, \\ \max_{j=1,\ldots,n} |l_j| > R-1}}  e^{-s q' w_*(|l|)} \\
		& \leq & c_M \int_{|x|>R-2} (1+|x|)^{-Mq'} \, dx.
\end{eqnarray*}
For $Mq' >n$ and by following the computations in the proof of Proposition 4.7 in \cite{rrs} we obtain
\[ \int_{|x|>R-2} (1+|x|)^{-Mq'} \, dx \leq C (R-1)^{n -Mq'} =: \widetilde{G}_R \]
with a constant $C = C(n,q,M)$ independent of $f$, $g$ and $R$. Thus, we get
\begin{eqnarray*}
	S_{1,t} & \lesssim & \widetilde{G}_R^{1/q'} \| f\|_{\mathcal{M}^{w_*}_{p_1,q}}  \| g\|_{\mathcal{M}^{w_*}_{p_2,q}} \qquad \text{and} \\
	S_{2,t} & \lesssim & \widetilde{G}_R^{1/q'} \| f\|_{\mathcal{M}^{w_*}_{p_1,q}}  \| g\|_{\mathcal{M}^{w_*}_{p_2,q}}
\end{eqnarray*}
as above. For arbitrary $N\in \N$ we can always choose $M\in \N$ such that $M- \frac{n}{q'} \geq N$. This yields the estimate
\[ \widetilde{G}_{R}^{1/q'} \leq C_N R^{-N} \]
for all $R\geq 2$ and with a constant $C_N$ depending on $n, q, M$ and $N$ but independent of $f$, $g$ and $R$.
\end{proof}

\begin{rem}
	Putting $w_* (|k|) = \log \langle k \rangle_{e^e} \log \log \langle k \rangle_{e^e}$ Proposition \ref{subalgebraUltraGenSV} coincides with Proposition 4.7 in \cite{rrs}.
\end{rem}

Now let us turn to the counterpart of Lemma \ref{estSuperUltraGenRV} for slowly varying weight functions.

\begin{lma} \label{estSuperUltraGenSV}
Let $1<p <\infty$ and  $1\le q \le \infty$.  Suppose $w_* \in \mathcal{W}_0 (\R)$. Then, for any $\vartheta >1$ and for any $N \in \N$
there exist  positive constants $b,c$ such that
\[ \| e^{iu}-1\|_{\mathcal{M}^{w_*}_{p,q}} \leq c \, \|u\|_{\mathcal{M}^{w_*}_{p,q}}\, 
\left\{
\begin{array}{lll}
 e^{\vartheta \, w_* (b\, \|u\|^{1+\frac 1N}_{\mathcal{M}^{w_*}_{p,q}})} &\qquad & \mbox{ if } \|u\|_{\mathcal{M}^{w_*}_{p,q}} > 1, \\
1 &&  \mbox{ if } \|u\|_{\mathcal{M}^{w_*}_{p,q}} \leq 1
\end{array}\right. 
\]
holds for all $u \in \mathcal{M}^{w_*}_{p,q}(\R^n)$.
In addition, the constant $b$ can be chosen independent of $\vartheta$ and $N$.
\end{lma}

\begin{proof}
Taking Proposition \ref{subalgebraUltraGenSV} into account we follow the proof of Lemma 4.8 in \cite{rrs}. 
\end{proof}

\begin{rem}
	The restriction to $p\in (1,\infty)$ is explained in the remark of Lemma \ref{estSuperUltraGenRV}. \\
	At this point we are able to compare the results of Lemma \ref{estSuperUltraGenRV} and Lemma \ref{estSuperUltraGenSV}. The methods of the proof of Lemma \ref{estSuperUltraGenSV} also work for modulation spaces $\mathcal{M}^{w_*}_{p,q} (\R^n)$ with $w_* \in \mathcal{W}_1 (\R)$. Thus, if we use weight functions $w_* \in \mathcal{W}_1 (\R)$ in both lemmas we easily see that Lemma \ref{estSuperUltraGenSV} yields a rougher estimate of the quantity $\| e^{i u} - 1 \|_{\mathcal{M}^{w_*}_{p,q}}$ than Lemma \ref{estSuperUltraGenRV}. This seems natural since the class $\mathcal{W}_1 (\R)$ contains only weight functions increasing faster than any function from the class $\mathcal{W}_0 (\R)$, i.e.
	\[ \bigcup_{w_* \in \mathcal{W}_1 (\R)} \mathcal{M}^{w_*}_{p,q} (\R^n) \subset \bigcap_{w_* \in \mathcal{W}_0 (\R) } \mathcal{M}^{w_*}_{p,q} (\R^n). \]
\end{rem}

Now we are in position to state the main theorem of this section. 

\begin{thm} \label{SuperpositionUltraGenSV}
Let $1<p <\infty$, $1\le q \le \infty$,  $\varepsilon >0$ and $\vartheta >1$. Suppose $w_* \in \mathcal{W}_0 (\R)$.
		Let $\mu$ be a complex measure on $\R$ such that
\begin{equation} \label{FourierEstUltraGenSV}
L_3(\lambda) := \int_{\R} e^{\vartheta\, w_*(\lambda |\xi|^{1+\varepsilon}) } \, d|\mu| (\xi) < \infty
\end{equation}
for any $\lambda >0 $ and such that $\mu(\R) = 0$. 
\\
Furthermore, assume that the function $f$ is the inverse Fourier transform of $\mu$. 
Then $f\in C^\infty (\R)$ and the composition operator $T_f: u \mapsto f \circ u$ maps $\mathcal{M}^{w_*}_{p,q} (\R^n)$ into $\mathcal{M}^{w_*}_{p,q} (\R^n)$.
\end{thm}
\begin{proof}
	The same arguments as in the proof of Theorem \ref{SuperpositionUltraGenRV} together with equation \eqref{FourierEstUltraGenSV} and Lemma \ref{estSuperUltraGenSV} yield the claim. 
\end{proof}

Note the following conclusion.

\begin{cor} \label{SuperpositionCorUltraGenSV}
	Let $1<p <\infty$, $1\le q \le \infty$ and $\varepsilon >0$. Let $w_* \in \mathcal{W}_0 (\R)$. Assume $\mu$ to be a complex measure on $\R$ with the corresponding bounded density function $g$, i.e., $d\mu (\xi) = g(\xi) \, d\xi$. Suppose that
		\begin{equation*} 
			\lim_{|\xi|\to \infty} \frac{w_*(|\xi|^{1+\varepsilon})}{\log |g(\xi)|} = 0
		\end{equation*}
		and $\displaystyle \int_{\R} d\mu (\xi) = \int_{\R} g(\xi) \, d\xi = 0$. 
If the function $f$ is the inverse Fourier transform of $g$, then $f\in C^\infty (\R)$ and the composition operator 
$T_f: u \mapsto f \circ u$ maps $\mathcal{M}^{w_*}_{p,q} (\R^n)$ into $\mathcal{M}^{w_*}_{p,q} (\R^n)$.
\end{cor}
\begin{proof}
Since $w_*$ is a slowly varying it follows that for sufficiently large $\xi$ it holds
\[ w_* (\lambda |\xi|) \leq (1+\delta) w_*(|\xi|) \]
for any $\lambda >0$ and any $\delta >0$. By choosing admissible $\lambda$ we can follow the proof of  Corollary 3.11 in \cite{rrs}. 
\end{proof}

\begin{ex} 
Let us devote ourselves again to the general weight function \eqref{WeightEx1} introduced in Section \ref{WeightClass}. If we assume $s=\infty$, then \eqref{WeightEx1} turns to a slowly varying functions. In order to justify then $w_* \in \mathcal{W}_0 (\R)$ we require either $r_1>1$ or $r_1=1$ and additionally the existence of $l\in\N_+$ such that $r_i=0$ for $i=2,3,\ldots, l$, $r_{l+1}>0$ and $r_{l+2}, r_{l+3}, \ldots, r_m \in \R$. Then \ref{A4} is satisfied as well as \ref{A3} and we have $w_* \in \mathcal{W}_0 (\R)$. Note that the other conditions of Definition \ref{DefClassWeight} are anyway satisfied by the representation \eqref{WeightEx1}. Let us mention particular examples. Taking the weight $w_* (|k|) = \log \langle k \rangle_{e^e} \log \log \langle k \rangle_{e^e}$ Theorem \ref{SuperpositionUltraGenSV} yields a non-analytic superposition operator on $\mathcal{M}^{w_*}_{p,q} (\R^n)$. This result coincides with Theorem 4.10 in \cite{rrs}. Moreover, due to Theorem \ref{SuperpositionUltraGenSV} we obtain non-analytic superposition results in the modulation spaces $\mathcal{M}^{w_*}_{p,q} (\R^n)$ if $w_* (|k|) = \log \langle k \rangle_e l_m \langle k \rangle_*$, $m \in \N$, and $w_* (|k|) = \big( \log \langle k \rangle_e \big)^\gamma$, $\gamma >1$, respectively.
\end{ex}

	\section{Open Questions and Concluding Remarks} \label{ConcRem}	
	
	The main concern of this paper has been to explain analytic as well as non-analytic superposition on modulation spaces equipped with a general class of ultradifferentiable weight functions. We found an appropriate class such that we cover a scale of weights increasing faster than any polynomial and at most subexponentially. However, in future work we still would like to get a better understanding of the borderline cases, that is, classical weighted modulation spaces $M^s_{p,q} (\R^n)$ equipped with Sobolev type weights and modulation spaces equipped with exponential type weights. \\
Within the scope of this work we are not able to treat the case of regularly varying weight functions $w_*$ of index $\alpha =1$ although we can assume that $\widetilde{w}_*$ is a decreasing slowly varying function according to representation \eqref{ws-51}. These types of modulation spaces $\mathcal{M}^{w_*}_{p,q} (\R^n)$ correspond to modulation spaces equipped with exponential type weights. However in this case we do not even obtain the algebra property. The proof of Theorem \ref{algebraUltraGen} fails because the conditions \ref{A7} and \ref{A8} cannot be satisfied for regularly varying functions of index $\alpha =1$. Also remark that condition \ref{A1} is not only a technical one. Taking the weight function $w_* (|k|) = \langle k \rangle$, which is regularly varying of index $\alpha = 1$, the authors in \cite{brs} have shown that $\mathcal{M}^{w_*}_{2,2} (\R^n)$ is not an algebra with respect to pointwise multiplication. Nevertheless the question arises if a decreasing slowly varying function in \eqref{ws-51} can provide some benefit in order to obtain multiplication results for modulation spaces $\mathcal{M}^{w_*}_{p,q} (\R^n)$ equipped with exponential type weights. \\
Note that there are several contributions to the algebra problem in weighted modulation spaces $M^s_{p,q} (\R^n)$, see \cite{feichtingerGroup}, \cite{iwabuchi}, \cite{sugimoto}, \cite{guo}, \cite{toftAlg}. \\
	We are able to apply superposition results to handle nonlinear partial differential equations by the concepts used in \cite{brs}. In future work we will study existence of local (in time) and global (in time) solutions of particular semi-linear Cauchy problems. \\
	Let us consider briefly the semi-linear Cauchy problem \eqref{GevreyNL} with initial data taken from $\mathcal{M}^{w_*}_{p,q}(\R^n)$. Assume that $a(t)$ has the modulus of continuity $\mu$, i.e.,
	\[ | a(t) - a(s) | \leq C \mu (|t-s|). \]
	For details we refer to \cite{lorenz}. We expect local (in time) existence results if the time-dependent coefficient $a=a(t)$ has the following moduli of continuity together with initial data $\phi, \psi$ taken from the corresponding modulation spaces $\mathcal{M}^{w_*}_{p,q} (\R^n)$:
	\begin{center}
	\begin{tabular}{|l|l|l|} \hline
		modulus of continuity & commonly called & weight function \\
		\hline
		\rule{0pt}{13pt} 	
 		$\mu (x) = x^\alpha$, $0< \alpha <1$ & H\"older-continuity & $w_* (|k|) = \langle k \rangle^{\frac{1}{s}}$, $s \leq \frac{1}{1-\alpha}$ \\
		\rule{0pt}{13pt} 	
		$\mu (x) = x( \log (\frac{1}{x}) + 1)$ & Log-Lip-continuity & $w_* (|k|) = \log \langle k \rangle_e$ \\
		\rule{0pt}{13pt} 	
 		$\mu (x) = x ( \log (\frac{1}{x}) + 1) l_m (\frac{1}{x}) $ & Log-Log$^{[m]}$-Lip-continuity & $w_* (|k|) = \log \langle k \rangle_e l_m \langle k \rangle_*$, $m \in \N$ \\[5pt]
		\hline
	 \end{tabular}
	 \end{center}
	 Consequently, if we take initial data $\phi, \psi \in \mathcal{M}^{w_*}_{p,q} (\R^n)$ with the wrong weight function $w_*$, i.e., a weight which is not chosen according to the assumed modulus of continuity on the coefficient $a=a(t)$, then we expect that the Cauchy problem \eqref{GevreyNL} becomes ill-posed. 
	 
\section*{Acknowledgement}
The author gratefully acknowledge the helpful advices and support of Prof. Sickel from Friedrich-Schiller University Jena. Moreover, the author wants to express his gratitude to Prof. Reissig from TU Bergakademie Freiberg for numerous constructive discussions.


\newpage
	\begin{thebibliography}{9}
		\bibitem{bgt} N. H. Bingham, C. M. Goldie, J. L. Teugels, \textit{Regular variation.} Encyclopedia of Mathematics and Its Applications, vol. 27., Cambridge University Press (1987)
		\bibitem{bjoerck} G. Bj\"orck, \textit{Linear partial differential operators and generalized distributions.} Arkiv f\"or Matematik \textbf{6.4} (1966), 351-407.
		\bibitem{bonet} J. Bonet, R. Meise, S. N. Melikhov, \textit{A comparison of two different ways to define classes of ultradifferentiable functions.} Bulletin of the Belgian Mathematical Society-Simon Stevin \textbf{14.3} (2007), 425-444.
		\bibitem{brs} G. Bourdaud, M. Reissig, W. Sickel, \textit{Hyperbolic equations, function spaces with exponential weights and Nemytskij operators.} Annali di Matematica Pura ed Applicata \textbf{182.4} (2003), 409-455.
		\bibitem{braun} R.W. Braun, R. Meise, B.A. Taylor, \textit{Ultradifferentiable functions and Fourier analysis.} Results in Mathematics \textbf{17.3-4} (1990), 206-237.
		\bibitem{colombini} F. Colombini, E. Janelli, S. Spagnolo, \textit{Well-posedness in the Gevrey classes of the Cauchy problem for a non-strictly hyperbolic equation with coefficients depending on time.} Ann. Sc. Norm. Super. Pisa, Cl. Sci., IV. Ser. \textbf{10} (1983), 291-312.
		\bibitem{cordero} E. Cordero, F. Nicola, \textit{Remarks on Fourier multipliers and applications to the wave equation.}  Journal of Mathematical Analysis and Applications \textbf{353.2} (2009), 583-591.
		\bibitem{feichtingerGroup} H. G. Feichtinger, \textit{Modulation spaces on locally compact Abelian group.}  Technical Report, University of Vienna (1983)
		\bibitem{gs} J. Galambos, E. Seneta, \textit{Regularly varying sequences.}  Proceedings of the American Mathematical Society \textbf{41.1} (1973), 110-116.
		\bibitem{groechenig} K. Gr\"ochenig, \textit{Foundations of Time-Frequency Analysis.}  Birkh\"auser (2001)
		\bibitem{guo} W. C. Guo, D. S. Fan, H. X. Wu, G. P. Zhao, \textit{Sharpness of some properties of weighted modulation spaces.}  Science China Mathematics (2015), 1-22.
		\bibitem{toftPseudo.}  K. Gr\"ochenig, J. Toft, \textit{Isomorphism properties of Toeplitz operators and pseudo-differential operators between modulation spaces.}  Journal d'Analyse Mathématique \textbf{114.1} (2011), 255-283.
		\bibitem{heil} C. Heil, \textit{Integral operators, pseudodifferential operators, and Gabor frames.}  Advances in Gabor analysis. Birkh\"auser Boston (2003), 153-169.
		\bibitem{iwabuchi} T. Iwabuchi, \textit{Navier-Stokes equations and nonlinear heat equations in modulation spaces with negative derivative indices.}  Journal of Differential Equations \textbf{248} (2009), 1972-2002.
		\bibitem{komatsu} H. Komatsu, \textit{Ultradistributions I. Structure theorems and a characterization.} J. Fac. Sci. Tokyo Sec. IA \textbf{20} (1973) , 25-105.
		\bibitem{lizorkin} P.I. Lizorkin, \textit{Multipliers of Fourier integrals and  bounds of convolution in spaces with mixed norms.} Applications. Math. Izvestiya \textbf{4} (1970),  225-255.
		\bibitem{lorenz} D. Lorenz, \textit{On the Well-Posedness of Hyperbolic Partial Differential Equations with Low-Regular, Time-Dependent Coefficients.} diploma thesis, Freiberg, TU Bergakademie Freiberg, 2015. 
		\bibitem{Ni} S.M. Nikol'skij, \textit{Approximation of functions of several variables and imbedding theorems.}  Springer, Berlin (1975)
		\bibitem{pascu} M. Pascu, \textit{On the definition of Gelfand-Shilov spaces.}  Annals of the University of Bucharest (mathematical series) \textbf{59.1} (2010), 125-133.
		\bibitem{qiang} H. Qiang, F. Dashan, C. Jiecheng, \textit{Cauchy problem for NLKG in modulation spaces with noninteger powers.}  arXiv preprint arXiv:1411.3535 (2014).
		\bibitem{rrs} M. Reich, M. Reissig, W. Sickel, \textit{Non-analytic Superposition Results on Modulation Spaces with Subexponential Weights.}  Accepted for publication in Journal of Pseudo-Differential Operators and Applications,  arXiv:1510.07521 (2015) 45 p.
		\bibitem{rodino} L. Rodino, \textit{Linear partial differential operators in Gevrey spaces.} World Scientific, Singapore (1993)
		\bibitem{ruzhansky} M. Ruzhansky, M. Sugimoto, B. Wang, \textit{Modulation spaces and nonlinear evolution equations.}  Springer Basel (2012), 267-283.
		\bibitem{schindl} G. Schindl, \textit{Exponential laws for classes of Denjoy-Carleman differentiable mappings.}  Dissertation, 2013.
		\bibitem{sugimoto} M. Sugimoto, N. Tomita, B. Wang, \textit{Remarks on nonlinear operations on modulation spaces.}  Integral transforms and special functions \textbf{22} (2011), 351-358.
		\bibitem{toft} J. Toft, \textit{Continuity properties for modulation spaces, with applications to pseudo-differential calculus - I.}  Journal of Functional Analysis \textbf{207} (2004), 399-429.
		\bibitem{toftConv} J. Toft, \textit{Convolution and Embeddings for Weighted Modulation Spaces.}  Birkh\"auser Basel (2004), 165-186.
		\bibitem{toftCubo} J. Toft, \textit{Pseudo-differential operators with smooth symbols on modulation spaces.}  Cubo - A Mathematical Journal (2009), 99-101.
		\bibitem{toftAlg} J. Toft, \textit{Sharp convolution and multiplication estimates in weighted spaces.}  Analysis and Applications \textbf{13.5} (2015), 457-480.
		\bibitem{triebel} H. Triebel, \textit{Theory of Function Spaces.}  Geest \& Portig K.-G. (1983)
		\bibitem{trzaa} H. Triebel, \textit{Modulation spaces on the Euclidean $n$-space.}  ZAA {\bf 2} (1983), 443-457.
		\bibitem{wang1} B. Wang, L. Han, C. Huang, \textit{Global well-posedness and scattering for the derivative nonlinear Schr\"odinger equation with small rough data.} Ann. I. H. Poincar{\'e} - AN  \textbf{26} (2009), 2253-2281.
		\bibitem{wangNonlinear} B. Wang, C. Huang, \textit{Frequency-uniform decomposition method for the generalized BO, KdV and NLS equations.}  Journal of Differential Equations \textbf{239.1} (2007), 213-250.
		\bibitem{wang} B. Wang, H. Hudzik, \textit{The global Cauchy problem for the NLS and NLKG with small rough data.}  Journal of Differential Equations \textbf{232} (2007), 36-73.
		\bibitem{wangExp} B. Wang, Z. Lifeng, G. Boling, \textit{Isometric decomposition operators, function spaces $E^\lambda_{p,q}$ and applications to nonlinear evolution equations.}  Journal of Functional Analysis \textbf{233.1} (2006), 1-39.
		\bibitem{zhong} Y. Zhong, J.C. Chen, \textit{Modulation space estimates for the fractional integral operators.}  Science China Mathematics \textbf{54.7} (2011), 1479-1489.
	\end{thebibliography}
\end{document}